\documentclass[11pt, reqno]{amsart}
\usepackage{amssymb,amsfonts}
\usepackage{amsbsy}
\usepackage{amsfonts}
\usepackage{amsmath}
\usepackage{amsthm}
\usepackage{amssymb}
\usepackage{ytableau}

\usepackage{amsxtra}

\PassOptionsToPackage{dvipsnames}{xcolor}
\usepackage{pgfplots}
\pgfplotsset{compat=1.11}
\usepackage{amsrefs}
\usepackage{tikz-cd}
\usepackage{mathrsfs}
\usepackage{enumerate}

\theoremstyle{plain} 
\newtheorem{thm}{Theorem}[section]

\newtheorem{prop}[thm]{Proposition}
\newtheorem{lem}[thm]{Lemma}

\theoremstyle{definition}
\newtheorem{defn}[thm]{Definition}

\theoremstyle{remark}
\newtheorem{rem}[thm]{Remark}

\numberwithin{equation}{section}
\numberwithin{figure}{section}

\newcommand{\DD}{{\rm D}}
\newcommand{\NN}{\mathbb{N}}
\newcommand{\PP}{\mathbb{P}}

\newcommand{\ZZ}{\mathbb{Z}}

\newcommand{\G}{{\overline{G}}}
\newcommand{\oH}{{\overline{\overline{G}}}}

\newcommand{\oA}{{{\overline{A}}}}
\newcommand{\oB}{{{\overline{B}}}}

\newcommand{\ord}{{\rm ord}}

\newcommand{\Gnu}{{\overline{\nu}}}
\newcommand{\oC}{\overline{C}}

\newcommand{\sI}{{\mathcal I}}
\newcommand{\pp} {{\mathfrak p}}
\newcommand{\sT} {{\mathscr T}}

\newcommand{\EE}{{\mathbb E}}

\definecolor {UMblue}  {RGB}{0, 39, 76}
\definecolor {UMmaize} {RGB}{255, 203, 5}

\definecolor {color_b}{RGB}{255,0,0}
\definecolor {color_c}{RGB}{20, 200, 30}
\definecolor {color_a}{RGB}{0,0,255}

\definecolor{lgreen} {RGB}{180,210,100}
\definecolor{dblue}  {RGB}{20,66,129}
\definecolor{ddblue} {RGB}{11,36,69}
\definecolor{lred}   {RGB}{220,0,0}
\definecolor{nred}   {RGB}{224,0,0}
\definecolor{norange}{RGB}{230,120,20}
\definecolor{nyellow}{RGB}{255,221,0}
\definecolor{ngreen} {RGB}{98,158,31}
\definecolor{dgreen} {RGB}{78,138,21}
\definecolor{nblue}  {RGB}{28,130,185}
\definecolor{jblue}  {RGB}{20,50,100}

\usepackage{epigraph}

\title[Products of extended binomial coefficients and their partial factorizations]{Products of extended binomial coefficients \\and their partial factorizations}

\author{Lara Du}
\address{Department of Mathematics, 
University of Michigan, Ann Arbor, MI 48109--1043, USA.
Current address: Draper Building, Berea College, 101 Chestnut Street, Berea, KY 40404, USA.}
\email{dul@berea.edu}
\author{Jeffrey Lagarias} 
\address{Department of Mathematics, University of Michigan, Ann Arbor, MI 48109--1043, USA.}
\email{lagarias@umich.edu}
\author {Wijit Yangjit}
\address{Department of Mathematics, University of Michigan, Ann Arbor, MI 48109--1043, USA.}
\email{yangjit@umich.edu}
\date{January 15, 2025, 11pt}

\begin{document}

\keywords{Binomial coefficients, Floor function, Sum of digits function, Running digit sum function}
\subjclass{11B65, 11B83, 11N37}

\begin{abstract}
This paper studies properties of the integer sequence $\overline{\overline{G}}_n=\prod_{k=0}^n\binom{n}{k}_{\mathbb{Z},\mathbb{N}}$ which is analogous to $\overline{G}_n=\prod_{k=0}^n\binom{n}{k}$, the product of the elements of the $n$-th row of Pascal's triangle. Here $\binom{n}{k}_{\mathbb{Z},\mathbb{N}}$ is an extended binomial coefficient, defined in the paper, constructed using an extended version of M. Bhargava's theory of generalized factorials. In 1996 M. Bhargava introduced a generalization of the factorial function, $n!_S=\prod_p\nu_n(S,p)$ in terms of their prime factorization, and defines associated binomial coefficients. The last two authors extended Bhargava's invariants further to define such invariants attached to each integer $b\ge2$. One obtains  extended factorials and extended binomial coefficients, and the maximal extension defines extended factorials $n!_{\mathbb{Z},\mathbb{N}}=\prod_{b\ge2}b^{\alpha_n(\mathbb{Z},b)}$ including all $b\ge2$, with associated extended binomial coefficients $\binom{n}{k}_{\mathbb{Z},\mathbb{N}}$, yielding $\overline{\overline{G}}_n$. We have $\overline{\overline{G}}_n=\prod_{b=2}^nb^{\overline{\nu}(n,b)}$ and the partial factorizations $\overline{\overline{G}}(n,x)=\prod_{b=2}^{\lfloor x\rfloor}b^{\overline{\nu}(n,b)}$. This paper shows $\log\overline{\overline{G}}(n,\alpha n)$ is well approximated by $f_{\overline{\overline{G}}}(\alpha)n^2\log n+g_{\overline{\overline{G}}}(\alpha)n^2$ as $n\to\infty$ for limit functions $f_{\overline{\overline{G}}}(\alpha)$ and $g_{\overline{\overline{G}}}(\alpha)$ defined for all $0\le\alpha\le1$. The remainder term has a power saving in $n$. The main results are deduced from study of functions $\overline{A}(n,x)$ and $\overline{B}(n,x)$ that encode statistics of the base $b$ radix expansions of the integer $n$ (and smaller integers), where the base $b$ ranges over all integers $2\le b\le x$. Unconditional estimates of $\overline{A}(n,x)$ and $\overline{B}(n,x)$ are derived.
\end{abstract}

\maketitle


%
%

\section{Introduction}

This paper studies 
extended factorials $n!_{\ZZ, \NN}$ and extended binomial coefficients $ \binom{n}{k}_{\ZZ, \NN}$,  
as defined below.
It studies the sequence given by the   product of extended binomial coefficients
\begin{equation}\label{eqn:extend-binom-prod}
\oH_n := \prod_{k=0}^n \binom{n}{k}_{\ZZ, \NN}.
\end{equation}
The extended binomial coefficients $\binom{n}{k}_{\ZZ, \NN}$
are a special case of  generalized  binomial coefficients
constructed using  an extension of Bhargava's theory of generalized factorials and
generalized binomial coefficients,  introduced  
by the last two authors in  \cite[Section 7]{LY:24a}.
(We  review Bhargava's theory and its extension in 
 Section \ref{subsubsec:212}.) 
We define
\begin{equation}
\binom{n}{k}_{\ZZ, \NN} := \frac{ n!_{\ZZ, \NN}}{k!_{\ZZ, \NN} (n-k)!_{\ZZ, \NN}},
\end{equation}
with the right side consisting of {\em extended factorials} $n!_{\ZZ, \NN}$.
The extended factorials are  in turn defined by  
\begin{equation}
n!_{\ZZ, \NN} :=\prod_{k=1}^n  [k]_{\ZZ, \NN},
\end{equation}
with the right side being a product of  {\em extended  integers} $[k]_{\ZZ, \NN}$.
Finally 
 the {\em extended integers} $[k]_{\ZZ, \NN}$ may be  defined  by  the formula
\begin{equation}
[n]_{\ZZ, \NN} := \prod_{b \mid n, b\ge2} b^{ \ord_b(n)},
\end{equation}
where $\ord_b(n)$ is the maximal $k\in\mathbb{N}$ such that $b^k$ divides $n$ (\cite[Theorem 7.3]{LY:24a}).
These generalized integers  $[n]_{\ZZ,\NN}$ 
have  an internal structure driven by the prime factorization of $n$,
and have large oscillations in their sizes for consecutive $n$.

According to  \cite[Theorem 7.6]{LY:24a}, the integer
$\oH_n$ is given explicitly by  a factorization formula:
\begin{equation}\label{eqn:oHn0} 
\oH_n = \prod_{b=2}^n b^{\Gnu(n,b)},
\end{equation} 
where  $\Gnu(n,b)$ are integers definable  in terms of the  base $b$ radix expansions
of  integers up to $n$. Specifically,  
\begin{equation}\label{eqn:gnu-b}
\Gnu(n,b):=\frac{2}{b-1}S_b(n)-\frac{n-1}{b-1}d_b(n),
\end{equation}
  where $d_b(n)$ is the sum of the base $b$ digits of $n$ and $S_b(n):=\sum_{j=1}^{n-1}d_b(j)$.
  The two terms on the right side of \eqref{eqn:gnu-b} are not necessarily integers, they
  are rational numbers with denominator dividing $b-1$, but their sum is an integer  and
   $\Gnu(n, b)=0$ for $b>n$; see Theorem \ref{thm:nub} below.

  This paper studies the asymptotic behavior  of the extended binomial products $\oH_n$ and  of its partial factorizations
  \begin{equation}\label{eqn:Hnx-def} 
\oH(n,x)  := \prod_{b=2}^{\lfloor x\rfloor} b^{\Gnu(n,b)},
\end{equation} 
where $1 \le x \le n$.
It extracts   a scaling limit as  $n \to \infty$
and $x= \alpha n$, where $0 < \alpha \le 1$ is fixed.

%

\subsection{Main results: Asymptotics of extended binomial products $\oH_n$} \label{sec:11}

The initial result  determines the growth rate  of the  full sum.

%
%

\begin{thm}\label{thm:oHn}
Let $\oH_n$ be given by 
\eqref{eqn:oHn0}.
Then for all integers $n\ge2$,
\begin{equation}\label{eqn:oH-aysmp1}
\log\oH_n=\frac{1}{2}n^2\log n+\left(\frac{1}{2}\gamma-\frac{3}{4}\right)n^2+O\left(n^{3/2}\log n\right),
\end{equation}
where $\gamma$ is Euler's constant.
\end{thm}

Theorem \ref{thm:oHn} has two main terms in the asymptotics followed by  a power-saving remainder term.
We note the appearance of Euler's constant in the second main term.
   The proof of Theorem \ref{thm:oHn}  is discussed
in Section \ref{sec:15a} and in Section \ref{sec:2aa}. 
It  obtains asymptotics  using the formula \eqref{eqn:gnu-b} by  directly estimating  various digit sums,  
as the  base $b$ is varied.

Theorem \ref{thm:oHn} is  used as  an initial condition  in obtaining growth estimates for partial factorizations $\oH(n,x)$
for general $1 \le x \le n$, which we discuss next.

%
%

\subsection{Main results: Asymptotics of partial factorizations $\oH(n,x)$}\label{sec:11b}
The main result of the paper  
determines the size of the partial factorization function $\oH(n,x)=\prod_{b=2}^{\lfloor x\rfloor}b^{\Gnu(n,b)}$ in the range $1\le x\le n$.
This ``partial factorization" is partial in two senses: it is partial in not being a factorization into primes, and it is partial in using
only the terms $b^{\Gnu(n,b)}$ for $b \le x$  in the full product. 
We establish  the following limiting behavior as $n\to\infty$ taking $x= x(n):=\alpha n$.

%
%

\begin{thm}\label{thm:oHnx-main}
Let $\oH(n,x)=\prod_{b=2}^{\lfloor x\rfloor}b^{\Gnu(n,b)}$. Then for all integers $n \ge 2$ and real $\alpha\in\left[\frac{1}{\sqrt{n}},1\right]$,
\begin{equation}\label{eqn:Gnx-main}
\log\oH(n,\alpha n)=f_\oH(\alpha)\,n^2\log n+g_\oH(\alpha)\,n^2+O\left(n^{3/2}\log n\right),
\end{equation}
in which:

  \emph{(a)} $f_{\oH} (\alpha)$  is a  function with $f_{\oH}(0)=0$ and defined for all $\alpha>0$ by
\begin{equation}\label{eqn:oHnx-parametrized-1}
f_{\oH}(\alpha)=\frac{1}{2}
+\frac{1}{2}\alpha^2\left\lfloor\frac{1}{\alpha}\right\rfloor^2+\frac{1}{2}\alpha^2\left\lfloor\frac{1}{\alpha}\right\rfloor-\alpha\left\lfloor\frac{1}{\alpha}\right\rfloor;
\end{equation}

\emph{(b)} $g_{\oH} (\alpha)$  is a  function with  $g_{\oH}(0)=0$ and  defined for all $\alpha>0$ by
\begin{eqnarray}
g_{\oH}(\alpha) &=&  \bigg( \frac{1}{2} \gamma - \frac{3}{4} \bigg) -\frac{1}{2}\bigg( H_{\lfloor\frac{1}{\alpha}\rfloor}- \log \frac{1}{\alpha} \bigg) 
 + \bigg(\log \frac{1}{\alpha} \bigg)\bigg( -\frac{1}{2} - \frac{1}{2} \alpha^2 \bigg\lfloor \frac{1}{\alpha} \bigg\rfloor \bigg\lfloor \frac{1}{\alpha} +1\bigg\rfloor + \alpha \bigg\lfloor \frac{1}{\alpha} \bigg\rfloor \bigg)  
 \nonumber \\
&&- \frac{1}{4} \alpha^2 \bigg\lfloor \frac{1}{\alpha} \bigg\rfloor \bigg\lfloor \frac{1}{\alpha} +1\bigg\rfloor + \alpha \bigg\lfloor \frac{1}{\alpha} \bigg\rfloor.\label{eqn:oHnx-parametrized-2}
\end{eqnarray}
Moreover, for all integers $n \ge2$ and real $\alpha \in\left[ \frac{1}{n},\frac{1}{\sqrt{n}}\right]$, 
\begin{equation}\label{eqn:oHna-bound2} 
\log\oH(n, \alpha n) = O \left( n^{3/2} \log n\right) .
\end{equation} 
\end{thm}

The  functions $f_{\oH}(\alpha)$ and $g_{\oH}(\alpha)$ are continuous functions of $\alpha$ on $[0,1]$
 which are analytic on $I_j= \left[\frac{1}{j+1}, \frac{1}{j}\right]$ for each
 $j \ge1$, and have breakpoints at $\alpha= \frac{1}{j}$ which are ``special values." 
 The asymptotic estimate \eqref{eqn:Gnx-main} implies that $f_{\oH}(\alpha)$ is obtained as  a  limit scaling function
$f_\oH(\alpha) :=\lim_{n\to\infty}\frac{1}{n^2\log n}\log\oH(n,\alpha n).$
 The function $f_{\oH}(\alpha)$ is given in Figure \ref{fig:AB1}.

\begin{figure}[h] 
\begin{center}
\includegraphics[scale=0.40]{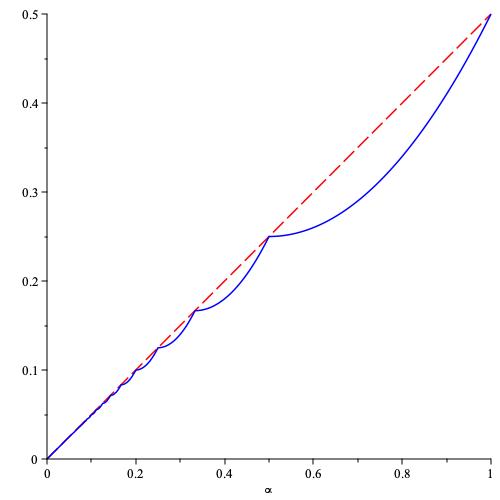}
\end{center}
\caption{Graph of the limit scaling function $f_{\oH}(\alpha)$,  $0 \le \alpha \le 1.$ 
Here $f_{\oH}(0)=0$ and $f_{\oH}(1)= \frac{1}{2}.$ The dashed line is $\tilde{f}(\alpha)= \frac{1}{2} \alpha$.}
\label{fig:AB1}
\end{figure} 

\newpage
 
The   function $f_{\oH}(\alpha)$ is identical with  a limit scaling  function $f_{G}(\alpha)$ 
obtained from the product of the usual  binomial coefficients,
 discussed in Section \ref{sec:201} below (cf.  \cite[Theorem 1]{DL:22}).
Since $f_{\oH}(\alpha)= f_{G}(\alpha)$,  the following  properties  
 are consequences of   \cite[Lemma 4.2]{DL:22}.
\begin{enumerate}
\item[(i)]The function $f_\oH(\alpha)$ is continuous on $[0,1]$. It is real-analytic on $(0,1)$ minus the points $\alpha=\frac{1}{2},\frac{1}{3},\frac{1}{4},\dots$. It has $f_{\oH}(0)=0$.
\item [(ii)]
The function $f_{\oH}(\alpha)$  is positive and strictly increasing on $(0,1]$. It has $f_{\oH}(1)= \frac{1}{2}.$
\item[(iii)]
There is a smooth  interpolating function
 $\tilde{f}(\alpha) =\frac{1}{2} \alpha$ on $[0,1]$  that satisfies
\begin{equation}\label{eqn:foH-bound}
f_{\oH} ( \alpha) \le \tilde{f}(\alpha)=  \frac{1}{2}\alpha  \quad \mbox{for all} \quad 0 \le \alpha \le 1.
\end{equation} 
The interpolation property is that $\tilde{f}(\alpha)=f_{\oH}(\alpha)$ at the
points $\alpha= 1,\frac{1}{2},\frac{1}{3},\dots$, and $\alpha=0$.  Equality holds nowhere else on $[0, 1]$. 
\item[(iv)] The function $f_{\oH}(\alpha)$ is   piecewise quadratic; i.e., for each integer $j \ge1$, it is  given by 
\begin{equation}
f_{\oH}(\alpha) = \frac{1}{2} - j\alpha + \frac{j(j+1)}{2}  \alpha^2 \quad \mbox{for all} \quad \frac{1}{j+1} \le \alpha \le \frac{1}{j}.
\end{equation} 
\end{enumerate} 
An attractive  alternate form for  $f_{\oH}(\alpha)$ is 
\begin{equation}\label{eqn:alt-foH}
f_{\oH}(\alpha)=\frac{1}{2}\alpha^2\left(\left\lfloor\frac{1}{\alpha}\right\rfloor+\left\{\frac{1}{\alpha}\right\}^2\right),
\end{equation}
where $\{x\}=x-\lfloor x\rfloor$ is the fractional part function. 

The  secondary  limit function $g_{\oH}(\alpha)$ is  
 pictured in Figure \ref{fig:AB2}.

\begin{figure}[h] 
\begin{center}
\includegraphics[scale=0.40]{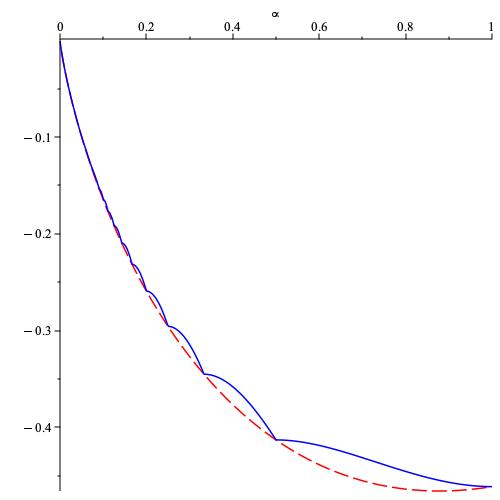}
\end{center}
\caption{Graph of the limit scaling function $g_{\oH}(\alpha)$, $0 \le \alpha \le 1$.
Here $g_{\oH}(0)=0$ and $g_{\oH}(1) = \frac{1}{2} \gamma - \frac{3}{4}  \approx - 0.46139$. 
The dashed curve is $\tilde{g}(\alpha)$.}
\label{fig:AB2}
\end{figure} 

Some properties of $g_{\oH}(\alpha)$ are: 
\begin{enumerate}
\item[(i)] The function $g_{\oH}(\alpha)$ is continuous on $[0,1].$ It is real-analytic on $(0,1)$ minus the points $\alpha= \frac{1}{2},\frac{1}{3},\frac{1}{4},\dots$. It has $g_{\oH}(0)=0$.
\item[(ii)] The function $g_{\oH}(\alpha)$ is  negative and strictly decreasing on $(0,1]$. 
It has $g_{\oH}(1) = \frac{1}{2} \gamma - \frac{3}{4}  \approx -0.46139$.
\item[(iii)] Let $\psi(x) = \frac{\Gamma'}{\Gamma}(x)$ be the digamma function. 
There is a smooth  interpolating function $\tilde{g}(\alpha)$
given by 
\begin{equation}
\tilde{g}(\alpha) = - \frac{1}{2} \psi\left(\frac{1}{\alpha}+1\right) + \frac{1-\alpha}{2}\log \frac{1}{\alpha} - \frac{\alpha}{4}\quad\text{for all}\quad0 < \alpha \le 1,
\end{equation}
which extends to a continuous function at  $\alpha=0$, with $\tilde{g}(0)=0$. It satisfies
\begin{equation}
g_{\oH}(\alpha) \ge \tilde{g}(\alpha)   \quad \mbox{for all} \quad 0 \le \alpha \le 1.
\end{equation}
The interpolation property  is that  $\tilde{g}(\alpha)=g_{\oH}(\alpha)$ at the
points $\alpha= 1,\frac{1}{2},\frac{1}{3},\dots$, and $\alpha=0$. Equality holds nowhere else on $[0, 1]$.
\item[(iv)] The function $\tilde{g}(\alpha)$ is strictly convex on $[0,1]$.
 The function $g_{\oH}(\alpha)$ is strictly concave on $ \left[\frac{1}{j+1}, \frac{1}{j} \right]$ for each integer $j \ge 2$.
(It is not concave on the interval $\left[ \frac{1}{2}, 1\right]$,  where it has a single inflection point.)
\end{enumerate}
These properties of $g_{\oH}(\alpha)$ are established in Lemma \ref{lem:93}.

The proof of Theorem \ref{thm:oHnx-main} is discussed in Section \ref{sec:15a} and Section \ref{sec:2aa} and
deduced from auxiliary results  in Sections \ref{sec:asymp-oBnx} and \ref{sec:asymp-oAnx}. 
Theorem \ref{thm:oHnx-main} is  a direct consequence of   Theorem \ref{thm:oHnx},
 for  $\oH(n,x)$, by the change of variable $\alpha= \frac{x}{n}$.

%
%

\subsection{Binomial products $\G_n$ and  partial factorizations $G(n,x)$}\label{sec:201}

To put the main results in perspective,    
we  review earlier  work on products of   binomial coefficients,
stating  results in parallel with Theorems \ref{thm:oHn} and \ref{thm:oHnx-main}.

The sequence $\G_n$ of the product of the binomial coefficients in the $n$-th row of Pascal's triangle,
\begin{equation}\label{eqn:defGn}
\G_n:=\prod_{k=0}^n\binom{n}{k}=\frac{(n!)^{n+1}}{\prod_{k=0}^n(k!)^2},
\end{equation}
 was studied in \cite{LagM:2016}.
 Here  $\G_n$ arises as  the reciprocal of the product of  nonzero unreduced Farey fractions of order $n$,
i.e., the reciprocal of the product of all  fractions $0 < \frac{k}{\ell} \le 1$, not necessarily in lowest terms,  having $1 \le k \le \ell \le n$.
Thus  for $n=6$, both $\frac{1}{3}$ and $\frac{2}{6}$ are included in the product. 

We  write   the  prime factorization of $\G_n$ as
\begin{equation}\label{eqn:productformula}
\G_n=\prod_{p}p^{\nu_p(\G_n)},
\end{equation}
where $\nu_p(a)$ denotes the additive $p$-adic valuation of $a$.
Theorem 5.1 of \cite{LagM:2016} shows that the 
quantities $\nu_p(\G_n)$ are given by a formula $ \Gnu(n,p)$ defined purely in terms of the base-$p$ radix
expansions of all the integers up to $n$, so
\begin{equation}\label{eqn:productformula2}
\G_n=\prod_{p}p^{\Gnu(n,p)},
\end{equation}
where $\Gnu(n,p)$ is given by  the base $p$ digit sum formula \eqref{eqn:gnu-b} restricted to $b=p$, a prime.
Such formulas for individual binomial coefficients are treated in \cite{Gra97}.

 One may easily derive, using Stirling's formula, the estimate
 \begin{equation}\label{eqn:logG-asymp}
\log\G_n=\frac{1}{2}n^2 +O \left( n \log n\right).
\end{equation}
 Moreover here is a complete asymptotic expansion (\cite[Theorem A.2]{LagM:2016}).
 %
 %
 \begin{thm}\label{thm:13} 
 There  is a full asymptotic expansion of $\log \G_n$, for each $M \ge 0$, 
 \begin{equation}\label{eqn:logG-asymp2}
\log\G_n=\frac{1}{2}n^2-\frac{1}{2}n\log n+\left(1- \frac{1}{2} \log(2\pi)\right) n - \frac{1}{3} \log n + \sum_{k=0}^M \frac{g_k}{n^k} + O_M \left( \frac{1}{n^{M+1}} \right).
\end{equation}
Here $g_0= -\frac{1}{2} \log (2\pi)- \frac{1}{12} + 2 \log A$, where $A$ is the Glaisher--Kinkelin constant, 
 and for all $k \ge 1$ the coefficients $g_k$ are rational numbers.
 \end{thm}
 This asymptotic expansion is based on  asymptotics of  the Barnes $G$-function. An asymptotic for    $\G_n$ itself was recently  given by Kellner \cite[Theorem 10]{Kellner:24}. 
 
In 2021 the first two authors \cite{DL:22} studied partial factorizations $G(n,x)$ of the product $\G_n$,
defined by 
\begin{equation}\label{eqn:Gnx-def}
G(n,x) = \prod_{p \le x} p^{\Gnu(n,p)}.
\end{equation}

Theorem 1 of  \cite{DL:22} stated: 
%
%
%
\begin{thm}\label{thm:Gnx-main} 
Let $G(n, x) = \prod_{p \le x} p^{\Gnu(n,p)}$. Then for  all integers $n \ge 4$ and  all $0< \alpha \le 1$,  
\begin{equation}\label{eqn:Gnx-main2}  
\log G(n,\alpha n)=f_G(\alpha)n^2+R_G(n,\alpha n), 
\end{equation}
where $f_G(\alpha)$ is a  function given for all $\alpha>0$ by
 \begin{equation}\label{eqn:Gnx-parametrized2}
f_G(\alpha)=\frac{1}{2}+\frac{1}{2}\alpha^2\left\lfloor\frac{1}{\alpha}\right\rfloor^2+\frac{1}{2}\alpha^2\left\lfloor\frac{1}{\alpha}\right\rfloor-\alpha\left\lfloor\frac{1}{\alpha}\right\rfloor,
\end{equation}
with $f_G(0)=0$ and $R_G(n, \alpha n)$ is a remainder term.

\emph{(1)} Unconditionally, there is a positive constant $c$ such that for all integers $n \ge 4$ and all  $0 < \alpha \le 1$, the remainder term satisfies
\begin{equation}\label{eqn:remU}
R_G(n,\alpha n)=O\left(\frac{1}{\alpha}n^2\exp(-c\sqrt{\log n})\right).
\end{equation}
The implied constant in the $O$-notation does not depend on $\alpha$.

\emph{(2)} Conditionally on the Riemann hypothesis, for  all integers $n \ge 4$ and  all $0 < \alpha \le 1$, the remainder term satisfies
\begin{equation}\label{eqn:remRH}
R_G(n,\alpha n)=O\left(\frac{1}{\alpha}n^{7/4}(\log n)^2\right).
\end{equation}
The  implied constant in the $O$-notation does not depend on $\alpha$.
\end{thm}

The limit scaling function 
$f_G(\alpha)= \lim_{n \to \infty} \frac{1}{n^2} \log G(n, \alpha n)$
is identical to $f_{\oH}(\alpha)$  pictured  in Figure \ref{fig:AB1}.

   The location of the zeros of the Riemann zeta function seem to control the remainder term $R_{G}(n, \alpha n)$.
 In the converse direction, the  paper \cite[Section 5]{DL:22} 
 observed that if one knew a good estimate at the  single point $\alpha= 1/2$,
 for the scaled asymptotics of the logarithm of the partial factorization of the central binomial coefficient 
 ${{2n}\choose {n}}$, taking only prime factors  $p \le n$, taking  the form
  $\log \prod_{p\le n}p^{\nu_p\left({2n\choose n}\right)} =  (2\log2-1)n +O\left( n^{1- \delta}\right)$   (i.e., taking $\alpha= \frac{1}{2}$)  
 for some $\delta >0$, then one could derive
  a zero-free region of the Riemann zeta function of the form $\Re(s)> 1- h(\delta)$
  for some $h(\delta)>0$.
  One may ask if a similar converse implication might  hold for binomial products at  $\alpha= \frac{1}{2}$;
  i.e., whether a bound 
  $\log G\left(n, \frac{1}{2}n\right)= \frac{1}{4}n^2 + O \left( n^{2- \delta}\right)$
 would imply  a nontrivial zero free region of similar shape.  (Here $f_{G}(1/2)= 1/4$.)

%
%

\subsection{Comparison}\label{sec:14a}

The study of various types of generalized binomial coefficients accentuates  
the unusually good  properties of the usual binomial coefficients, 
viewed inside a larger framework. One  fruitful generalization is to  
the Gaussian polynomials ${n\brack m}_q=\frac{(q)_n}{(q)_m (q)_{n-m}}$, with $(q)_k= \prod_{i=1}^k(1-q^i)$
and $0 \le m \le n$,  also called $q$-binomial coefficients,
are a $q$-generalization which specialize to binomial coefficients as $q \to 1$.
They are  constructed by an additive recursion,
treated in \cite[Section 3.3]{Andrews:98}, and in Berndt \cite[Section 4]{Berndt:10}, who 
relates them to hypergeometric series.

The  extended binomial coefficients $\binom{n}{k}_{\ZZ, \NN}$ 
treated here have a purely multiplicative definition.
This paper determines the large scale behavior of these coefficients, given by
the overall growth rate of $\log \oH_n$ and the 
 partial factorizations $\log \oH(n,x)$, with $x = \alpha n$
as $n \to \infty$. It shows smooth behavior described by limit scaling functions,  with a power-savings remainder term.
These properties of the limit scaling functions are 
``emergent properties" on a large scale, in the sense that the 
individual extended integers $[n]_{\ZZ, \NN}$, from which 
extended binomial coefficients are built, show large oscillations in size, 
as $n$ varies.

 The nice  ``local"  and additive properties of the usual binomial coefficients fail to hold for extended binomial coefficients.
\begin{enumerate}
\item[(i)] 
On the level of a single row, the extended binomial coefficients lack the montonicity and log-concavity properties, 
of the usual binomial coefficients. The coefficients on the $n$-th row 
 are  sometimes not  unimodal. 
 \item[(ii)]
 The  sum of  extended binomial coefficients in the $n$-th row of the 
extended Pascal's triangle is not an increasing function of $n$.
\item[(iii)]
There is no known additive analogue of the two term additive recursion $\binom{n}{k} = \binom{n-1}{k} + \binom{n-1}{k-1}$
in Pascal's triangle.
\end{enumerate}

On the positive side,
 the asymptotics of $\log \oH(n,\alpha n)$ given in Theorem \ref{thm:oHnx-main} has, 
after  the two main terms,    an unconditional  remainder term exhibiting a  power-savings $O(n^{3/2 + \varepsilon})$
in the full region $\alpha \in \left[\frac{1}{n}, 1\right]$, $n\ge2$.
The summation over all bases $b \ge 2$ has smoothed the remainder term,
avoiding any appeal to the Riemann hypothesis. 

One may hope to use $\oH(n,x)$ (with two variables $n$ and  $x$),
perhaps adding additional statistics that allow only $b$'s in arithmetic progressions,
 to permit  more direct recovery of information about the binomial products $G(n,x)$.
 It is also an  interesting question to account for the appearance of the secondary limit scaling function $g_{\oH}(\alpha)$
 in the asymptotics of $\log \oH(n,\alpha n)$ from the viewpoint of prime number theory.
These topics  are  briefly addressed in Section \ref{sec:8}.

%
%

\subsection{Proofs}\label{sec:15a}

The proofs of Theorems \ref{thm:oHn} and \ref{thm:oHnx-main}  use 
the radix expansion formula \eqref{eqn:gnu-b}.
These proofs use  the regularity and structure  to the variation of the  base $b$ radix expansions of {\em fixed}  $n$, { when 
 the base $b \ge 2$ is varied.} The digit sum variables for fixed $n$  but varying bases  $b$  are significantly correlated, 
 running over $b \le n$, and are perfectly correlated for $b \ge n$.
 The limit scaling functions  quantify the  nature of the cross-correlation in different ranges of the
two variables.

The proof approach to Theorem \ref{thm:oHnx-main} parallels that  for estimating partial factorizations
 of  products of binomial coefficients in \cite{DL:22}.
There is  however an important difference between  radix expansions 
appearing in  the binomial coefficient case and those appearing in the
extended binomial coefficient case.  It is that  for prime $p$ and integers $n_1, n_2$ one has 
\begin{equation}\label{nu-equality}
\nu_p(n_1) + \nu_p(n_2) = \nu_p(n_1n_2),
\end{equation}
while for composite $b$ one has only
\begin{equation}\label{eq:nu-inequality}
\nu_b(n_1) + \nu_b(n_2) \le \nu_b(n_1n_2).
\end{equation}
For $\oH(n,x)$  the extra averaging over all integers $b\in(1,x]$ 
leads to  the unconditional power savings $O\left( n^{3/2} \log n \right)$ in the
remainder term.  The  form  of the secondary term $g_{\oH}(\alpha)$ 
in the asymptotics  seems  in part  related to   the inequality 
\eqref {eq:nu-inequality} for composite $b$.

  In Section \ref{sec:2aa}   we outline the proof approach to
proving the main theorems,  which has four auxiliary theorems. 
The analysis splits the logarithm into separate  contributions of the two
 radix expansion terms on the right side of  \eqref{eqn:gnu-b}.
  It determines  asymptotics of the
 two terms separately; they give rise to several new scaling
 functions, given in Section \ref{sec:2aa}.

%
%

\subsection{Contents}\label{sec:26aa}

\begin{enumerate}
\item
Section \ref{sec:2a} reviews prior work on generalized factorials
and on digit expansions. 
\item
Section \ref{sec:2aa} outlines
 the proof method for the main result. 
It  is based on subsidiary estimates  for functions
$\oA(n,x)$ and $\oB(n,x)$  obtained using the
additive digit sum factorization. As  initial conditions
one needs estimates for 
$\oA(n) = \oA(n,n)$ and $\oB(n)= \oB(n,n)$.
\item
Section \ref{sec:2} collects facts about digit sums and
gives estimates for  auxiliary 
sums needed in later estimates. In particular it treats the bivariate sums
\begin{equation}
\oC(n,x):=\sum_{1 \le b\le x}\left\lfloor\frac{n}{b}\right\rfloor\log b.
\end{equation}
\item
Section \ref{sec:oB} estimates $\oB(n)$, proving  Theorem \ref{thm:oBn}.
\item
Section \ref{sec:AH} estimates $\oA(n)$, proving Theorem \ref{thm:oAn}.
The main  Theorem \ref{thm:oHn} for $\oH_n$ follows combining $\oA(n)$
and $\oB(n)$. 
\item
Section \ref{sec:asymp-oBnx} estimates $\oB(n,x)$, proving  Theorem \ref{thm:oBnx}.
 Theorem \ref{thm:oBnx-cor} for $\oB(n,\alpha n)$ immediately follows. 
\item
Section \ref{sec:asymp-oAnx} estimates $\oA(n,x)$, proving Theorem \ref{thm:oAnx}. 
 Theorem \ref{thm:oAnx-cor} for $\oA(n,\alpha n)$ immediately  follows. 
\item
Section \ref{sec:oHnx} estimates $\oH(n,x)$, proving  Theorem \ref{thm:oHnx}.
The main Theorem \ref{thm:oHnx-main} for $\oH(n, \alpha n)$ follows. 
\item
Section \ref{sec:8} presents concluding remarks. These include 
a discussion of  extensions of the construction of ``multiple integral" integer sequences
$\G_n^{(-j)}$ for $j \ge 1$.
\end{enumerate}

%
%

\section{Prior work}\label{sec:2a}

%
%

\subsection{Generalized factorials and generalized binomial coefficients}\label{sec:10}

\subsubsection{ Bhargava's  generalized factorials}\label{subsubsec:211}

In the late 1990's M. Bhargava \cite{Bhar:97a} introduced a notion of generalized factorials and generalized
binomial coefficients associated to arbitrary subsets $S$ of a Dedekind domain $\DD$,
based on a notion of {\em $\pp$-orderings} of subsets $S \subseteq \DD$,
which are  defined for all nonzero prime ideals $\pp$ of the Dedekind domain. 
 The generalized factorials were defined  as integral ideals 
 in terms of associated {\em $\pp$-invariants}, $\nu_n(S, \pp) := \pp^{\alpha_n(S, \pp)}$  as
$$
n!_{S} = \prod_{\pp} \pp^{\alpha_n(S, \pp)}.
$$
Bhargava's definition of generalized factorials gives them in terms of their prime  factorization.
For the general definition of the $\pp$-invariants $\nu_n(S, \pp)$ 
we  refer the reader to \cite{Bhar:97a} or \cite{Bhar:00}. (The notation $\alpha_n(S, \pp)$ for
the exponent is introduced in \cite{LY:24a}, and in  \cite{LY:24b} for the general case.)

 Bhargava related these  invariants to rings of integer-valued  polynomials, and 
 obtained many applications (\cite{Bhar:09}). For connections with integer-valued polynomials,
 see \cite{CC: 96} and \cite{CC:16}.
 The generalized binomial coefficients were defined (as fractional ideals of $\DD$) by
 $ {{n}\choose{k}}_S = \frac{n!_{S}}{k!_{S} (n-k)!_{S}},$
and  they are provably integral ideals.

Bhargava treated  the special case  $\DD=\ZZ$ in \cite{Bhar:00}.
For the ring  $\ZZ$ 
the  prime ideals in the definition 
can be identified with the nonnegative primes, and the generalized factorials and  binomial coefficients
are then nonnegative integers.
Bhargava notes that the choice  of set $S=\ZZ$ recovers the usual factorials and the usual binomial coefficients,
see \cite[Definition 7]{Bhar:00}. 

\subsubsection{ Bhargava's factorials, extended}\label{subsubsec:212}
In 2024 the  last two authors (in \cite{LY:24a})  extended  Bhargava's notion of $p$-orderings on $\ZZ$
  to define $b$-orderings  for all ideals $(b)$ of $\ZZ$, with associated
invariants $\alpha_n(S, b)$.
 In this setting one can define, for all $n \ge 0$   generalized factorials
$$
n!_{S, \sT} := \prod_{b \in \sT} b^{\alpha_n(S, b)}
$$
depending on two parameters $S$ and $\sT$, 
in which  $S$ is a nonempty subset of $\ZZ$,
and $\sT\subseteq \NN$ is an arbitrary nonempty set of ideals of $\ZZ$. 
On identifying the monoid   of  ideals $\sI(\ZZ)$ of $\ZZ$ under ideal multiplication
 with the monoid  $\NN$ of nonnegative
integers under multiplication, these factorials are
 provably nonnegative integers. 

One  can define generalized
binomial coefficients 
$$
\binom{n}{k}_{S, \sT} := \frac{   n!_{S, \sT} } {k!_{S, \sT} (n-k)!_{S, \sT} };
$$
they are also nonnegative  integers, see \cite[Lemma 6.2]{LY:24a}.

 Bhargava's factorials correspond to 
the special choice $\sT= \PP$,
the set of all nonzero primes. The simplest generalization of 
the ordinary factorial is  produced by the full 
set $S =\ZZ$, and the maximal choice $\sT=\NN$.
In  this case the  contribution of  the zero ideal $0$ and the unit ideal $(1)$ to the resulting  extended factorial
are always  factors of $1$  (whenever $\sT$ is an infinite set), so they may be disregarded,
and one restricts the product to the set of $b \ge 2$. 

The extreme case $S=\ZZ$ and $\sT= \NN$ 
giving the particular extended  factorials
$ n!_{\ZZ, \NN}$ and extended  binomial coefficients $\binom{n}{k}_{\ZZ,\NN}$
was studied in Section 7 of  \cite{LY:24a}.

The generalized factorials $ n!_{\ZZ, \NN}$  fit in the general framework  of Knuth and Wilf \cite{KW89} treating generalized factorials and binomial coefficients as products of generalized integers (denoted $C_n$ in their paper). 
The sequence  of generalized integers $[n]_{\ZZ,\NN}$  
is  however not a regularly divisible sequence, a main focus of  \cite{KW89}.

We  note that our  terminology  ``extended binomial coefficients"   $\binom{n}{k}_{\ZZ, \NN}$ differs from 
the ``extended binomial coefficients",  denoted $\binom{n}{k}^{(q)}$,  defined  by  
 $
 \sum_{k=0}^{\infty} \binom{n}{k}^{(q)} x^k = (1+x^2 + \cdots + x^q)^n,
 $
  in the combinatorial literature,  cf. Comtet \cite[p. 77]{Comtet:74} and  \cite{Neuschel:14}.

%
%

\subsection{Radix  expansions to base $b$} \label{sec:231}

Work of the second author and Mehta \cite{LagM:2016} in 2016,
 studied radix expansion statistics   
 which hold the integer $n$ fixed, 
while varying across different radix bases up to $n$,
in connection with products of binomial coefficients,
 e.g., statistics  of the type of $A(n)$ and  $B(n)$.
 (Motivation for this work  originally came from study  of   
  analogous statistics for products of Farey fractions, given in \cite{LagM:2017}.) 
The paper \cite{DL:22} of the first two authors then studied the statistics $A(n,x)$ and $B(n,x)$
for partial factorizations of products of binomial coefficients. 
The motivation of \cite{DL:22} was study of prime number distribution from  a novel direction.

There  has been a great deal of study of the radix statistics $d_b(n)$ and $S_b(n)$
for a fixed base $b \ge 2$ and letting $n$ vary.
Work on $d_b(n)$ has mainly been probabilistic, for random integers
in an initial interval $[1,n]$, which is surveyed by  Chen et al \cite{CHZ:2014}.
One has for all $n \ge 1$,
\begin{equation} 
\EE[ d_b(k): 0 \le k \le n-1] \le \frac{b-1}{2} \log_b n,
\end{equation} 
 a result which is close to sharp when $n = b^k$ for some $k \ge 1$.
 We have 
 $d_b(n) \le (b-1) \log_b (n+1)$;
 see Lemma \ref{lem:dbn-Sbn-bound}. 
 It implies 
 $$ \oB(n)\le\sum_{b=2}^n\frac{n-1}{b-1}\left(\frac{(b-1)\log(n+1)}{\log b}\right)\log b=(n-1)^2\log(n+1).$$

Work on the smoothed function  $S_b(n)$ 
studying the asymptotics
as $n \to \infty$ started with 
Bush \cite{Bush:40} in 1940, followed by Bellman and  Shapiro \cite{BelS:48}, 
and Mirsky \cite{Mir:49}, who in 1949 showed for fixed $b \ge 2$, the asymptotic formula 
$$
S_b(n) = \frac{b-1}{2} n \log_b(n) +O(n).
$$ 
In 1952 Drazin and Griffith \cite{DG52} 
deduced an inequality implying 
\begin{equation}\label{eqn:S-ineq} 
S_b(n) \le \frac{b-1}{2} n \log_b n,
\end{equation}
for all $b \ge 2$ and $n \ge 1$; see Lemma \ref{lem:dbn-Sbn-bound}.
This inequality is sharp: equality holds at $n=b^k$ for all $k \ge 1$;
see \cite[Theorem 5.8]{LagM:2016}. 
Using Drazin and Griffith's  inequality \eqref{eqn:S-ineq} for $S_b(n)$ we have 
 \begin{equation} \label{eqn:A-inequality} 
 \oA(n) \le  \sum_{b=2}^n \frac{2}{b-1} \left(\frac{(b-1)n\log n}{2\log b}\right) \log b 
 =  n(n-1) \log n.
 \end{equation}
 
A formula of Trollope \cite{Tro:68} in 1968 gave an exact formula for $S_b(n)$ for base $b= 2$. 
Notable work of  Delange \cite{Del:1975} obtained exact formulas for $S_b(n)$ for all $b \ge 2$,
which exhibited an oscillating term in the asymptotics. 
We mention later work of  Flajolet et al \cite{FGKPT94} and Grabner and Hwang \cite{GH05}. 
Recently Drmota and Grabner \cite{DrmGra10}  surveyed this topic.

There are other local inequalities satisfied by
 the functions $S_b(n)$. 
In 2011 Allaart \cite[Equation (4)]{Allaart:11}
showed
an approximate convexity inequality for binary expansions: for all $0 \le \ell \le m$, 
\begin{equation}
S_{2}(m+ \ell) + S_{2}(m-\ell) - 2 S_{2}(m) \le \ell.
\end{equation} 
Allaart \cite[Theorem 3]{Allaart:17}
proved a generalization to any base $b$:
for all $0\le k \le m$,
\begin{equation}
S_{b}(m+ k) + S_{b}(m-k) - 2 S_b(m)  \le \left\lfloor \frac{b+1}{2} \right\rfloor k.
\end{equation} 
Allaart \cite[Theorem 1]{Allaart:17}
proved a superadditivity inequality valid for base $b$ expansions: for all $m,n \ge 1$,  
\begin{equation}
S_b(m+n) \ge S_b(m) + S_b(n) + \min(m,n).
\end{equation}

%
%

\section{Proof method for  the main result: Auxiliary theorems}\label{sec:2aa}

We outline the proof method for Theorem \ref{thm:oHnx-main}, which 
relies on analysis of  digit expansions.  There are four auxiliary theorems used
to obtain the result.

%
%

\subsection{Results: Asymptotics of $\oA(n)$ and $\oB(n)$}\label{sec:13}

Set 
\begin{equation}\label{eqn:oA-function}
\oA(n,x)=\sum_{2\le b\le x}\frac{2}{b-1}S_b(n)\log b
\end{equation}
and
\begin{equation}\label{eqn:oB-function}
\oB(n,x)=\sum_{2\le b\le x}\frac{n-1}{b-1}d_b(n)\log b.
\end{equation}
The proof is based on  the identity
\begin{equation}\label{eqn:oHABx}
\log \oH(n,x)  = \oA(n, x) - \oB(n, x).
\end{equation}
The identity is obtained by taking logarithms of both sides of the product formula \eqref{eqn:Hnx-def}
for $\oH(n,x)$ and substituting the formula \eqref{eqn:gnu-b}  
for each $\Gnu(n,b)$. 

The  resulting functions $\oA(n, x)$ and $\oB(n, x)$ are arithmetical sums
that combine 
behavior of the base $b$ digits of the integer $n$, viewing $n$ as fixed, and {\em varying the radix base  $b$}.
These functions  are weighted averages of statistics 
of the radix expansions of $n$ for varying bases $2\le b\le x$. 
The interesting range of $x$  is  $1 \le x \le n $  because $\oH(n,x)$ ``freeze" at $x=n$: 
$\oH(n,x)=\oH(n,n)$ for all $x \ge n$.

The main part of the paper determines the asymptotics of two nonnegative arithmetic functions $\oA(n,x)$ 
and $\oB(n,x)$, from which we  obtain asymptotics for $\log \oH(n,x)$ via \eqref{eqn:oHABx}.
 Most previous work on individual radix statistics have studied holding the radix base $b$ fixed and varying $n$, 
reviewed in  Section \ref{sec:231}.
But here  the radix sums $\oA(n,x)$ and $\oB(n,x)$ hold $n$ fixed and vary the radix base $b$.

The  proofs  first obtain   estimates for the  special case $x=n$, setting 
\begin{equation}\label{eqn:oA-function-0}
\oA(n):=\oA(n,n)=\sum_{b=2}^n\frac{2}{b-1}S_b(n)\log b,
\end{equation}
\begin{equation}\label{eqn:oB-function-0}
\oB(n):=\oB(n,n)=\sum_{b=2}^n\frac{n-1}{b-1}d_b(n)\log b.
\end{equation}
We determine  asymptotics separately for  the two functions $\oA(n)$ and $\oB(n)$  as $n \to \infty$, giving
a main term and  a bound on the remainder term. The analysis first estimates the fluctuating term $\oB(n)$
depending on $d_b(n)$, and then estimates $\oA(n)$ using the estimates for  $\oB(j)$ for $ 1 \le j \le n$.

%
%

\begin{thm}\label{thm:oBn}
Let $\oB(n)$ be given by \eqref{eqn:oB-function-0}.
Then for all integers $n\ge2$,
\begin{equation}\label{eqn:oB-asymp1}
\oB(n)=(1-\gamma)n^2\log n+\left(\gamma+\gamma_1-1\right)n^2+O\left(n^{3/2}\log n\right),
\end{equation}
where $\gamma$ is Euler's constant
and $\gamma_1$ is the first Stieltjes constant.
\end{thm}

Two constants  appear in this formula, which are
the Stieltjes constants $\gamma_0=\gamma$  (also known as the Euler--Mascheroni constant)
and $\gamma_1\approx  -0.07282$.
The {\em Stieltjes constants} $\gamma_n$  appear in the Laurent expansion of the Riemann zeta function at $s=1$, as
\begin{equation} 
\zeta (s) = \frac{1}{s-1} + \sum_{n=0}^{\infty} \frac{ (-1)^n }{n!} \gamma_n (s-1)^n. 
\end{equation}
Here $\gamma_0= \gamma \approx 0.57722$ is Euler's constant, and more generally
\begin{equation} \label{eqn:stieltjes}
\gamma_m:=\lim_{n\to\infty}\bigg(\sum_{k=1}^n\frac{(\log k)^m}{k}-\frac{(\log n)^{m+1}}{m+1}\bigg).
\end{equation}

To establish Theorem \ref{thm:oBn} we first show
the main contribution in the sum  $\oB(n)$ comes from  those  values  $b$
having  $b > \sqrt{n}$, whose key  property  is that 
their base $b$  radix expansions have {\em exactly two digits}. 
Proceeding in a similar fashion to \cite{DL:22},
summing over all two-digit patterns,   we obtain a formula of shape  
$\oB(n) = (1-\gamma)n^2 \log n +Cn^2+ O( n^{2- \delta})$ with an
unconditional  power-saving remainder term
possible because the sums involved are over all integers $2\le b \le n$
rather than over all primes $p \le n$.

We also  deduce a corresponding result for $\oA(n)$.

%
%

\begin{thm}\label{thm:oAn}
Let $\oA(n)$ be given by \eqref{eqn:oA-function-0}.
Then for all integers $n\ge2$,
\begin{equation}\label{eqn:oA-aysmp1}
\oA(n)=\left(\frac{3}{2}-\gamma\right)n^2\log n+\left(\frac{3}{2}\gamma+\gamma_1-\frac{7}{4}\right)n^2+O\left(n^{3/2}\log n\right),
\end{equation}
where $\gamma$ is Euler's constant
and $\gamma_1$ is the first Stieltjes constant. 
\end{thm} 

In \cite{DL:22} the analogue of  Theorem \ref{thm:oAn} could be
obtained immediately from Theorem \ref{thm:oBn} 
using knowledge of $\log \G_n$ available using Stirling's formula.
Such information is not available for $\log \oH_n$ and an entirely new method is 
used in Section \ref{sec:AH} to obtain the result via Theorem \ref{thm:oBn}.

Theorem \ref{thm:oHn} follows directly from Theorems \ref{thm:oBn} and \ref{thm:oAn},
via 
\begin{equation}\label{eqn:oHAB} 
\log \oH_n = \oA(n) - \oB(n),
\end{equation}
which is a special case of  \eqref{eqn:oHABx}, taking $x=n$.

The constants appearing in the main term of the asymptotics 
of  $\oA(n)$ and $\oB(n)$
give quantitative information on 
cross-correlations between the statistics $d_b(n)$ and $S_b(n)$  
of the base $b$ digits of $n$ (and smaller integers) as the base $b$ varies while $n$ is held fixed. 
The occurrence of Euler's constant in the main term of these asymptotic estimates 
indicates subtle arithmetic behavior in these sums, even though the defining sums are taken  over all $b$ and not over primes $p$; cf. the survey \cite{Lag:13}.

%
%

\subsection{Results: Asymptotics of $\oA(n,x)$ and $\oB(n,x)$}\label{sec:14} 

We first  determine  asymptotics for  $\oB(n, \alpha n)$ for $0 < \alpha \le 1$,
by bootstrapping  the result for $\oB(n)=\oB(n,n)$ decreasing $x$ from $x=n$.
In what follows $H_m =\sum_{j=1}^m \frac{1}{j}$
and $J_m=\sum_{j=1}^m \frac{\log j}{j}$.

%
%

\begin{thm}\label{thm:oBnx-cor}
Let $\oB(n,x)=\sum_{b=2}^{\lfloor x\rfloor}\frac{n-1}{b-1}d_b(n)\log b$. Then for all integers 
$n \ge 2$ and  real 
$\alpha\in \left[ \frac{1}{\sqrt{n}} , 1\right]$,
\begin{equation}\label{eqn:Bnx-main}  
\oB(n, \alpha n) = f_{\oB}(\alpha)  n^2 \log n + g_{\oB}(\alpha)n^2 + O \left( n^{3/2} \log n  \right),
\end{equation}
in which:
\begin{enumerate}
\item[\emph{(a)}] $f_{\oB}(\alpha)$ is a function with $f_{\oB}(0)=0$ and defined for all $\alpha>0$ by 
 \begin{equation}\label{eqn:oBnx-parametrized1} 
f_{\oB}(\alpha) =  (1- \gamma)+ \left( H_{\lfloor \frac{1}{\alpha}\rfloor}- \log \frac{1}{\alpha} \right)  - \alpha \left\lfloor \frac{1}{\alpha}\right\rfloor;
\end{equation}
\item[\emph{(b)}] $g_{\oB}(\alpha)$ is a  function with  $g_{\oB}(0)=0$ and defined for all $\alpha>0$ by
\begin{eqnarray}
g_{\oB}(\alpha) &=&\left(\gamma+\gamma_1-1 \right)- 
 \left( H_{\lfloor \frac{1}{\alpha}  \rfloor}  - \log \frac{1}{\alpha} \right) - \left( J_{\lfloor\frac{1}{\alpha}\rfloor} - 
 \frac{1}{2} \left(\log \frac{1}{\alpha}\right)^2 \right)  \nonumber\\
 &&+ \left(\log \frac{1}{\alpha}\right) \left(-1+\alpha \left\lfloor \frac{1}{\alpha} \right\rfloor  \right)   + \alpha\left\lfloor \frac{1}{\alpha} \right\rfloor.\label{eqn:oBnx-parametrized2} 
\end{eqnarray} 
\end{enumerate} 
Moreover, for all integers $n \ge2$ and real $\alpha \in\left[ \frac{1}{n},\frac{1}{\sqrt{n}}\right]$,
\begin{equation}\label{eqn:oBna-bound2} 
\oB(n, \alpha n) = O \left( n^{3/2} \log n \right) .
\end{equation} 
\end{thm} 

The formulas  given for $f_{\oB}(\alpha)$ and $g_{\oB}(\alpha)$ are sums of functions that are discontinuous 
at the points $\alpha=1,\frac{1}{2},\frac{1}{3},\dots$. However  $f_{\oB}(\alpha)$ and $g_{\oB}(\alpha)$ 
are continuous functions of $\alpha$, extending to a limit value at the endpoint $\alpha=0$. 
The limiting values are $f_{\oB}(0) =0$ and $g_{\oB}(0) =0$.   \medskip

The function $f_{\oB}(\alpha)$  is pictured in Figure \ref{fig:B3}. It is the same function as $f_{B}(\alpha)$ in Theorem 1.5 in \cite{DL:22}.
 Some properties of this limit function are: 
\begin{enumerate}
\item[(i)] $f_{\oB}(\alpha)$ is continuous but not differentiable on $[0,1].$ 
It is real-analytic on $(0,1)$ minus the points $\alpha= \frac{1}{2},\frac{1}{3},\frac{1}{4},\dots$.
\item[(ii)] $f_{\oB}(\alpha)$ is strictly increasing  on $[0,1]$. 
It has $f_{\oB}(0)=0$ and $f_{\oB}(1) = 1-\gamma\approx  0.42278$.
\end{enumerate}

\begin{figure}[h]
\includegraphics[scale=0.40]{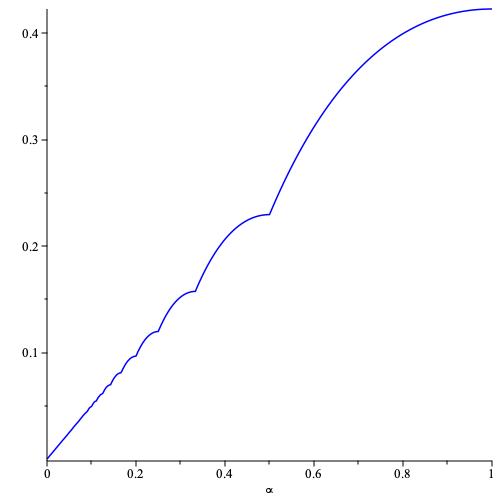}
\caption{Graph of the limit scaling function $f_{\oB}(\alpha)$,  $0 \le \alpha \le 1$. 
Here $f_{\oB}(0)=0$ and $f_{\oB}(1)=1-\gamma\approx0.42278$.
}
\label{fig:B3}
\end{figure}

\newpage 

The function $g_{\oB}(\alpha)$  is pictured in Figure \ref{fig:B4}. Some properties of this limit function are: 
\begin{enumerate}
\item[(i)] $g_{\oB}(\alpha)$ is continuous but not differentiable on $[0,1].$ 
It is real-analytic on $(0,1)$ minus the points $\alpha= \frac{1}{2},\frac{1}{3},\frac{1}{4},\dots$.
\item[(ii)] $g_{\oB}(\alpha)$ is strictly decreasing on $[0,1]$. 
It has $g_{\oB}(0)=0$ and $g_{\oB}(1) = \gamma +\gamma_1 -1 \approx -0.49560$.
\end{enumerate}

\begin{figure}[h] 
\includegraphics[scale=0.40]{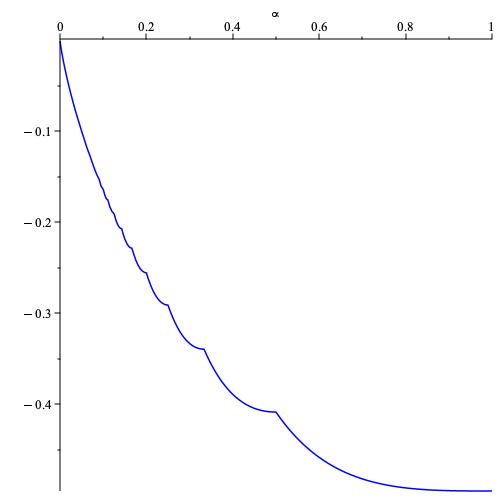}
\caption{Graph of the limit scaling function $g_{\oB}(\alpha)$,  $0 \le \alpha \le 1$. 
Here $g_{\oB}(0)=0$ and $g_{\oB}(1)=\gamma+\gamma_1-1\approx -0.49560$.
}
\label{fig:B4}
\end{figure}

Turning to $\oA(n,x)$, we obtain its asymptotics using a recursion starting from $\oA(n,n)$ 
and working downward, given by \eqref{eqn:oAnn-oAnx-formula}, which involves a
 function $\oC(n,x)$  studied in Proposition \ref{prop:23n}. 
This recursion is different from that used in \cite{DL:22}, which started from $A(x,x)$
and worked upward using $B(y,x)$ for $x < y<n$.

%
%

\begin{thm}\label{thm:oAnx-cor}
Let $\oA(n, x) = \sum_{b=2}^{\lfloor x\rfloor}\frac{2}{b-1} S_b(n) \log b.$ 
Then for all integers
 $n \ge 2$ and real $\alpha\in\left[\frac{1}{\sqrt{n}},1\right]$,
\begin{equation}\label{eqn:Anx-main}  
\oA(n, \alpha n) = f_{\oA}(\alpha)  n^2 \log n + g_{\oA}(\alpha) n^2 + O\left( n^{3/2}\log n\right),
\end{equation}
in which:
\begin{enumerate} 
\item[\emph{(a)}]     $f_{\oA}(\alpha)$  is a function with $f_\oA(0)=0$ and defined for all $\alpha>0$ by
 \begin{equation}\label{eqn:Anx-parametrized} 
 f_{\oA}(\alpha) = \left(\frac{3}{2} - \gamma\right)+  \left( H_{\lfloor \frac{1}{\alpha}\rfloor}- \log \frac{1}{\alpha} \right) + \frac{1}{2} \alpha^2 \left\lfloor \frac{1}{\alpha}\right\rfloor^2 + \frac{1}{2} \alpha^2 \left\lfloor \frac{1}{\alpha}\right\rfloor
- 2 \alpha \,\left\lfloor \frac{1}{\alpha} \right\rfloor;
\end{equation}
\item[\emph{(b)}]  $g_{\oA}(\alpha)$  is a function with $g_{\oA}(0)=0$ and defined  for all $\alpha>0$ by
\begin{eqnarray}
g_{\oA}(\alpha) &= & \left(  \frac{3}{2} \gamma + \gamma_1 -\frac{7}{4}\right) - \frac{3}{2} \left( H_{\lfloor \frac{1}{\alpha} \rfloor} - \log\frac{1}{\alpha}  \right)
- \left( J_{\lfloor \frac{1}{\alpha} \rfloor} \nonumber   - \frac{1}{2} \left(\log \frac{1}{\alpha}\right)^2 \right)\nonumber\\
&& +\left(\log \frac{1}{\alpha}\right)\left( -\frac{3}{2} -\frac{1}{2} \alpha^2  \left\lfloor \frac{1}{\alpha}\right\rfloor \left\lfloor \frac{1}{\alpha} +1 \right\rfloor +2\alpha \left\lfloor \frac{1}{\alpha} \right\rfloor \right) \nonumber\\
&& -\frac{1}{4} \alpha^2 \left\lfloor \frac{1}{\alpha}\right\rfloor \left\lfloor \frac{1}{\alpha} +1 \right\rfloor + 2 \alpha \left\lfloor \frac{1}{\alpha} \right\rfloor.\label{eqn:oAnx-parametrized} 
 \end{eqnarray}
 \end{enumerate} 
Moreover, for all integers $n \ge2$ and real $\alpha \in\left[ \frac{1}{n},\frac{1}{\sqrt{n}}\right]$, 
\begin{equation}\label{eqn:oAna-bound2} 
\oA(n, \alpha n) = O \left( n^{3/2} \log n  \right) .
\end{equation} 
\end{thm} 

The implied constant in the $O$-notation does not depend on  $\alpha$.
The formulas for $f_{\oA}(\alpha)$ and $g_{\oA}(\alpha)$ are a sum of functions that are discontinuous 
at the points $\alpha= 1,\frac{1}{2},\frac{1}{3},\dots$. 
However the functions $f_{\oA}(\alpha)$ and $g_{\oA}(\alpha)$ 
are continuous functions of $\alpha$.

The function $f_{\oA}(\alpha)$  is pictured in Figure \ref{fig:A5}.

\begin{figure}[h]
\includegraphics[scale=0.4]{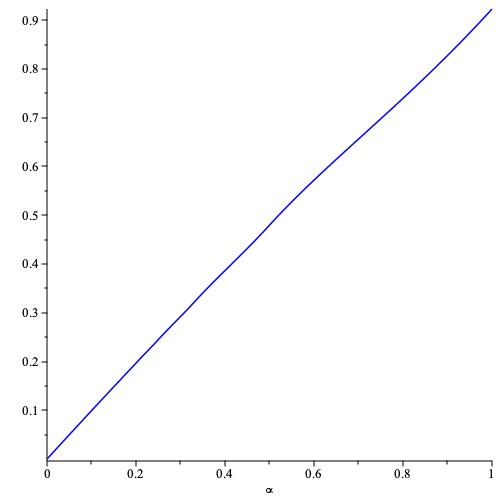}
\caption{Graph of the limit scaling function $f_{\oA}(\alpha)$,  $0 \le \alpha \le 1$. 
Here $f_{\oA}(0)=0$ and $f_{\oA}(1)=\frac{3}{2}-\gamma\approx0.92278$.
}
\label{fig:A5}
\end{figure}

The function $f_{\oA}(\alpha)$ is the same function as $f_{A}(\alpha) $ in Theorem 1.6 in \cite{DL:22}. Some properties of this limit function are:
\begin{enumerate}
\item[(i)] $f_{\oA}(\alpha)$ is continuous on $[0,1].$  It has a continuous derivative on $(0,1)$, given by
$$
f_{\oA}'(\alpha)=\alpha\left(\left\lfloor\frac{1}{\alpha}\right\rfloor+\left\{\frac{1}{\alpha}\right\}^2\right).
$$
It is real-analytic on $(0,1)$ minus the points $\alpha= \frac{1}{2},\frac{1}{3},\frac{1}{4},\dots$.
\item[(ii)] $f_{\oA}(\alpha)$ is strictly increasing  on $[0,1]$. It has $f_{\oA}(0)=0$ and $f_{\oA}(1) = \frac{3}{2} -\gamma\approx0.92278$.
\end{enumerate}

The function $g_{\oA}(\alpha)$  is pictured in Figure \ref{fig:A6}.

\begin{figure}[h]
\includegraphics[scale=0.40]{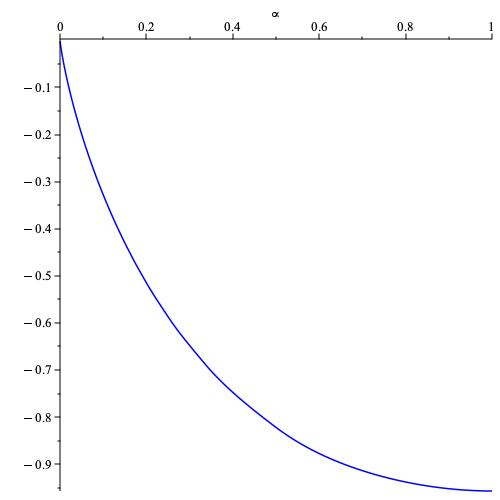}
\caption{Graph of the limit scaling function $g_{\oA}(\alpha)$,  $0 \le \alpha \le 1$. 
Here $g_{\oA}(0)=0$ and $g_{\oA}(1)=\frac{3}{2}\gamma+\gamma_1-\frac{7}{4}\approx -0.95699$.
}
\label{fig:A6}
\end{figure} 

Some properties of the limit function $g_{\oA}(\alpha)$ are: 
\begin{enumerate}
\item[(i)] $g_{\oA}(\alpha)$ is continuous on $[0,1]$. It has a continuous derivative on $(0,1)$, given by
$$
g_{\oA}'(\alpha)=-\alpha\left(\log\frac{1}{\alpha}\right)\left(\left\lfloor\frac{1}{\alpha}\right\rfloor+\left\{\frac{1}{\alpha}\right\}^2\right).
$$
It is real-analytic on $(0,1)$ minus the points $\alpha=\frac{1}{2},\frac{1}{3},\frac{1}{4},\dots$.
\item[(ii)] $g_{\oA}(\alpha)$ is strictly decreasing on $[0,1]$. It has $g_{\oA}(0)=0$ and $g_{\oA}(1) = \frac{3}{2} \gamma+\gamma_1- \frac{7}{4}  \approx -0.95699$.
\end{enumerate}

We obtain  Theorem \ref{thm:oHnx-main} as a corollary of the 
two previous theorems, substituting their estimates into the formula
$$\log \oH(n,x)  = \oA(n, x) - \oB(n, x).$$
At $x=n$ Euler's constant cancels out of the main term of order $n^2 \log n$,
but it is still present in the secondary term of order $n^2$ in
Theorem \ref{thm:oHnx-main}.

To summarize, the  analytic details of proving Theorems \ref{thm:oHn} and \ref{thm:oHnx-main} are  more involved
than those of \cite{DL:22}. On the one hand, the  sums  over all $b\ge2$ are 
 easier to handle than sums over primes in \cite{DL:22}, and  the extra averaging over $b$
 leads to unconditional results.
However   the  additional  limit functions become more complicated, and are 
given as sums of many more discontinuous functions which must be combined properly  
to obtain  continuous limit functions.

%
%

\section{Preliminaries}\label{sec:2}

The first subsection establishes properties of radix expansion statistics $\Gnu(n, b)$, and 
derives inequalities on the size of $d_b(n)$ and $S_b(n)$.  
The next four subsections estimate four families of
sums for an integer $n$ and a real number $x$, treated as step functions: the harmonic numbers $H(x) = \sum_{b=1}^{\lfloor x\rfloor} \frac{1}{b},$
the sums $J(x) = \sum_{b=1}^{\lfloor x\rfloor} \frac{\log b}{b}$, the sums $\oC(n,x) = \sum_{b=1}^{\lfloor x\rfloor} \lfloor \frac{n}{b} \rfloor \log b$, and
$L_i(n) = \sum_{b=2}^n b (\log b)^i$ for $ i \ge 1$.

%
%

\subsection{Radix expansion statistics} \label{subsec:21nn}

Fix an integer $b\ge2$. Let $n$ be a positive integer. Then $n$ can be written uniquely as
\begin{equation}\label{eqn:base-b}
n=\sum_{i=0}^ka_i(b,n)b^i,
\end{equation}
where $a_i(b, n)\in\{0,1,2,\dots,b-1\}$ are the \emph{base-$b$ digits} of $n$ and the \emph{top} digit $a_k$ is positive. We say that $n$ has $k+1$ digits in base $b$. One has $b^k\le n<b^{k+1}$. Hence the number of base-$b$ digits of $n$ is
$$
\left\lfloor \frac{\log n}{\log b}\right\rfloor+1.
$$
Each base-$b$ digit of $n$ can also be expressed in terms of the floor function:
\begin{equation}\label{eqn:floor-recursion}
a_i(b,n)=\left\lfloor\frac{n}{b^i}\right\rfloor-b\left\lfloor\frac{n}{b^{i+1}}\right\rfloor.
\end{equation}
Note that \eqref{eqn:floor-recursion} also defines $a_i(b,n)$ to be $0$ for all $i>\frac{\log n}{\log b}$. The following two statistics of the base-$b$ digits of numbers will show up frequently in this paper.


  \begin{defn}
  (1) The {\em sum of digits function} $d_b(n)$ is given by
  \begin{equation}\label{eqn:dbn}
  d_b(n):=\sum_{i=0}^{\left\lfloor\frac{\log n}{\log b}\right\rfloor}a_i(b,n)=\sum_{i=0}^\infty a_i(b,n),
  \end{equation}
  where $a_i(b,n)$ is given by \eqref{eqn:floor-recursion}.\\
   
  (2) The {\em running digit sum function} $S_b(n)$ is given by
   \begin{equation}\label{eqn:Sbn}
   S_b(n) := \sum_{j=1}^{n-1} d_b(j).
   \end{equation}
   \end{defn}
   
%
%

\begin{thm}\label{thm:nub}
Let $b\ge2$ be an integer, and let  the radix expansion statistic $\Gnu(n,b)$ given for all integers $n \ge 1$ by  
\begin{equation}\label{eqn:localformula}
\Gnu(n,b) =\frac{2}{b-1}S_b(n)-\frac{n-1}{b-1}d_b(n).
\end{equation}
Then:
\begin{enumerate}
\item[\emph{(1)}] For all integers $n \ge 1$, 
$\Gnu(n,b)$ is  a nonnegative integer.
\item[\emph{(2)}]
 $\Gnu(n,b)=0$ if and only if $n=ab^k+b^k-1$ for some $a\in\{1,2,3,\dots,b-1\}$ and integer $k \ge0$.
 \end{enumerate} 
\end{thm} 

\begin{proof}
To show (1),  we substitute \eqref{eqn:floor-recursion} into \eqref{eqn:dbn} and obtain
\begin{equation}\label{eqn:dbn-formula}
d_b(n)=\sum_{i=0}^\infty\left\lfloor\frac{n}{b^i}\right\rfloor-b\sum_{i=0}^\infty\left\lfloor\frac{n}{b^{i+1}}\right\rfloor=n-(b-1)\sum_{i=1}^\infty\left\lfloor\frac{n}{b^i}\right\rfloor.
\end{equation}
We then substitute \eqref{eqn:dbn-formula} into \eqref{eqn:Sbn} and obtain
\begin{equation}\label{eqn:Sbn-formula}
S_b(n)=\sum_{j=1}^{n-1}j-(b-1)\sum_{j=1}^{n-1}\sum_{i=1}^\infty\left\lfloor\frac{j}{b^i}\right\rfloor=\frac{n(n-1)}{2}-(b-1)\sum_{i=1}^\infty\sum_{j=1}^{n-1}\left\lfloor\frac{j}{b^i}\right\rfloor.
\end{equation}
Now, we substitute \eqref{eqn:dbn-formula} and \eqref{eqn:Sbn-formula} into \eqref{eqn:localformula} and obtain 
\begin{align}
\Gnu(n,b)&=\bigg(\frac{n(n-1)}{b-1}-2\sum_{i=1}^\infty\sum_{j=1}^{n-1}\left\lfloor\frac{j}{b^i}\right\rfloor\bigg)-\bigg(\frac{n(n-1)}{b-1}-(n-1)\sum_{i=1}^\infty\left\lfloor\frac{n}{b^i}\right\rfloor\bigg)\nonumber\\
&=\sum_{i=1}^\infty\bigg((n-1)\left\lfloor\frac{n}{b^i}\right\rfloor-2\sum_{j=1}^{n-1}\left\lfloor\frac{j}{b^i}\right\rfloor\bigg)\nonumber\\
&=\sum_{i=1}^\infty\sum_{j=1}^{n-1}\left(\left\lfloor\frac{n}{b^i}\right\rfloor-\left\lfloor\frac{j}{b^i}\right\rfloor-\left\lfloor\frac{n-j}{b^i}\right\rfloor\right).\label{eqn:Gnunb-formula}
\end{align}
The last quantity \eqref{eqn:Gnunb-formula} expresses $\Gnu(n,b)$ as the sum of integers, which are all nonnegative
due to the  identity valid for all real $x$ and $y$, 
$$
\lfloor x +y\rfloor =  \lfloor x \rfloor + \lfloor y \rfloor + \lfloor  \{ x\} + \{ y\} \rfloor  \ge  \lfloor x \rfloor +  \lfloor y \rfloor,  
$$
see Graham et al \cite[Section 3.1, page 70]{GKP:94}.
Hence $\Gnu(n,b)$ is a nonnegative integer.

We show (2). We 
 prove the `only if' part first. Suppose that $n$ is a positive integer not of the form $ab^k+b^k-1$, where $1\le a\le b-1$ and $k\ge0$.
 Then $cb^\ell\le n\le(c+1)b^\ell-2$ for some $c\in\{1,2,3,\dots,b-1\}$ and positive integer $\ell$.
 We show $\Gnu(n,b)$ is positive. 
 We see that the double sum in  \eqref{eqn:Gnunb-formula},
 is greater than or equal to the summand with $(i,j)=\left(\ell,b^\ell-1\right)$. It follows that 
$$
\Gnu(n,b)\ge\left\lfloor\frac{n}{b^\ell}\right\rfloor-\left\lfloor\frac{b^\ell-1}{b^\ell}\right\rfloor-\left\lfloor\frac{n-b^\ell+1}{b^\ell}\right\rfloor=c-0-(c-1)=1.
$$
Thus, if $\Gnu(n,b)=0$, then $n$ must be of the form $ab^k+b^k-1$ with $1\le a\le b-1$ and $k\ge0$.

Conversely, suppose that $n$ is of the form $ab^k+b^k-1$ with $1\le a\le b-1$ and $k\ge0$. Suppose that $j$ is an integer with $1\le j\le n-1$. For $i\le k-1$, we have $a_i(b,j)\le b-1=a_i(b,n)$.
For $i\ge k$, we also have $a_i(b,j)\le a_i(b,n)$ because $j<n$. Hence 
$$
a_i(b,n-j)=a_i(b,n)-a_i(b,j)
$$
for all $i\ge0$. Summing over $i\ge0$, we obtain
$$
d_b(n-j)=d_b(n)-d_b(j).
$$
Summing over $1\le j\le n-1$, we obtain
$$
S_b(n)=(n-1)d_b(n)-S_b(n),
$$
which implies
$$
\Gnu(n,b)=\frac{2}{b-1}S_b(n)-\frac{n-1}{b-1}d_b(n)=0.
$$
This completes the proof.
\end{proof}


\begin{rem}\label{rem:23}
In general, $\Gnu(n,b)$ does not equal the largest integer $k$ such that $b^k$ divides $\G_n$, which we denote by
$\nu_b(\G_n)$. Moreover $\Gnu(n,b)$ can  be  larger or smaller than $\nu_b(\G_n)$. For example,   $\Gnu(4,4)=3>2=\nu_4(\G_4)$, while  $\Gnu(6,4)=1<2=\nu_4(\G_6)$.
\end{rem}

We establish  inequalities on the size of $d_b(n)$ and $S_b(n)$.

%
%

\begin{lem}\label{lem:dbn-Sbn-bound}
For all integers $b\ge2$ and $n\ge1$, we have
\begin{equation}\label{eqn:dbnineq}
1\le d_b(n)\le\frac{(b-1)\log(n+1)}{\log b},
\end{equation}
\begin{equation}\label{eqn:Sbnineq}
0\le S_b(n)\le\frac{(b-1)n\log n}{2\log b}.
\end{equation}
\end{lem}

\begin{proof}
The lower bound in \eqref{eqn:dbnineq} follows from the observation that $d_b(n)$ is greater than or equal to the top (base-$b$) digit of $n$, which is at least $1$. The lower bound in \eqref{eqn:Sbnineq} then follows from the positivity of $d_b(j)$. 

The upper bound in \eqref{eqn:Sbnineq} is a result of Drazin and Griffith \cite[Theorem~1]{DG52}. To prove the upper bound in \eqref{eqn:dbnineq}, we apply Theorem \ref{thm:nub}:
$$
0\le(b-1)\Gnu(n,b)=2S_b(n)-(n-1)d_b(n)=2S_b(n+1)-(n+1)d_b(n).
$$
On replacing $n$ by $n+1$ in \eqref{eqn:Sbnineq}, we obtain $S_b(n+1)\le\frac{(b-1)(n+1)\log(n+1)}{2\log b}$. Hence
$$
d_b(n)\le\frac{2}{n+1}S_b(n+1)\le\frac{(b-1)\log(n+1)}{\log b},
$$
as desired.
\end{proof}

%
%

\subsection{The harmonic numbers $H_n$}  \label{subsec:22nn}

For  positive real numbers $x \ge 1$,  we consider the step function
$$
H(x):=\sum_{1\le b\le x}\frac{1}{b}.
$$
At integer values
$n= \lfloor x \rfloor$ we write $H(x)=H_{\lfloor x\rfloor}= H_n$, the $n$-th harmonic number.

%
%

\begin{lem}\label{lem:21n}
For all positive integers $n$, we have
\begin{equation}\label{eqn:harmonic}
H_n = \log n + \gamma + \frac{1}{2n} + O\left( \frac{1}{n^2} \right),
\end{equation}
where $\gamma \approx 0.57721$ is Euler's constant.
\end{lem}

\begin{proof}
This standard result appears in Tenenbaum \cite[Chapter I.0, Theorem 5]{Ten15}. 
\end{proof}

The restriction to integer $n$  is needed in Lemma \ref{lem:21n} because  for positive real numbers $x$, one has
$$
H(x)  - \log x - \gamma = \Omega_{\pm} \left( \frac{1}{x} \right).
$$
Indeed, using the Euler--Maclaurin summation formula \cite[Theorem~B.5]{MV07}, one can show that for real numbers $x\ge1$, 
$$
H(x)=\log x+\gamma+\frac{1-2\{x\}}{2x}+O\left(\frac{1}{x^2}\right).
$$
Hence, $\limsup_{x\rightarrow\infty}x(H(x)-\log x-\gamma)=\frac{1}{2}$ and
$\liminf_{x\rightarrow\infty}x(H(x)-\log x-\gamma)=-\frac{1}{2}$.\smallskip

%
%

\subsection{Estimates: $J(x)$}\label{subsec:22nn2}

For real numbers $x \ge 1$, we consider the step function
\begin{equation}\label{eqn:Jx} 
J(x):=\sum_{1 \le b\le x}\frac{\log b}{b}. 
\end{equation} 
At  integer values
$n= \lfloor x \rfloor$ we write $J(x)=J_{\lfloor x\rfloor}=J_n$.
The asymptotics of this step function of $x$  involve the first Stieltjes constant $\gamma_1$,
defined in  Section \ref{sec:13}.

%
%

\begin{lem}\label{lem:27}
For all real numbers $x\ge1$, we have
\begin{equation}\label{eqn:Jx1}
J(x)=\frac{1}{2}(\log x)^2+\gamma_1+O\left(\frac{\log(x+1)}{x}\right),
\end{equation}
where $\gamma_1 \approx- 0.0728158$  is the first Stieltjes constant.
\end{lem}

\begin{proof}
By partial summation, we obtain
\begin{equation}\label{PartialSum}
J(x)=\sum_{1 \le b\le x}\frac{\log b}{b}=(\log x)H(x)-\int_1^x\frac{H(u)}{u}\,du.
\end{equation}
It is well-known that $H(u)=\log u+\gamma+R(u)$, where the remainder $R(u)\ll\frac{1}{u}$, for $u\ge1$. (See \cite[Corollary~1.15]{MV07}.) 
On inserting this in \eqref{PartialSum} and rearranging, we get
\begin{align*}
J(x)&=\frac{1}{2}(\log x)^2  + (\log x) R(x)     - \bigg( \int_1^\infty\frac{R(u)}{u}\,du-\int_x^\infty\frac{R(u)}{u}\,du\bigg) \\
&=\frac{1}{2}(\log x)^2+c+O\left(\frac{\log(x+1)}{x}\right),
\end{align*}
where $c:=-\int_1^\infty\frac{R(u)}{u}\,du$. By taking $x\rightarrow\infty$, we see that
$$c=\lim_{x \to \infty} \left(J(x) - \frac{1}{2} (\log x)^2\right) = \lim_{x\to \infty}\bigg(\sum_{b\le x}\frac{\log b}{b}-\frac{1}{2}(\log x)^2\bigg)=\gamma_1,$$
the first Stieltjes constant, according to  \eqref{eqn:stieltjes}.
\end{proof}

%
%

\subsection{Estimates: $\oC(n,x)$}  \label{subsec:22nn3}

\par For  real numbers $n\ge 1$ and $x \ge 1$, let
\begin{equation}\label{eqn:oC-defn}
\oC(n,x):=\sum_{1 \le b\le x}\left\lfloor\frac{n}{b}\right\rfloor\log b.
\end{equation}
Here, $\oC(n,x)$  is a nonnegative step function of the real variable $x$,
viewing $n$ as fixed.
This function stabilizes for  $x \ge n$:
\begin{equation}\label{eqn:oCstab}
\oC(n,x)=\oC(n,n)\quad\mbox{for}\quad x\ge n.
\end{equation}

%
%

\begin{prop}\label{prop:23n}
\emph{(1)} For all real numbers $n\ge2$, we have
$$
\oC(n,n)=\frac{1}{2}n(\log n)^2+(\gamma-1)n\log n+(1-\gamma)n+O\left(\sqrt{n}\log n\right).
$$

\emph{(2)}  For all real numbers $n\ge2$ and $x$ such that $1\le x\le n$, we have
\begin{align}\label{eqn:oCn-diff}
\oC(n,n)-\oC(n,x)&=\int_x^n\left\lfloor\frac{n}{u}\right\rfloor\log u\,du+O\left(\frac{n\log n}{x}\right).
\end{align}
In addition,
\begin{equation} \label{eqn:oCn-integral} 
\int_x^n \left\lfloor\frac{n}{u}\right\rfloor\log u\,du = 
\left(H_{\left\lfloor\frac{n}{x}\right\rfloor}-\frac{x}{n}\left\lfloor\frac{n}{x}\right\rfloor\right)(n\log n-n) 
- \left( J_{\left\lfloor\frac{n}{x}\right\rfloor}- \frac{x}{n}\left\lfloor\frac{n}{x}\right\rfloor \log\frac{n}{x}
\right)n.
\end{equation}
\end{prop}

To prove Proposition \ref{prop:23n} we use the following identity. 

%
%

\begin{lem}\label{lem:oCfe}
For all real numbers $n \ge 1$ and $x \ge 1$, we have
\begin{eqnarray}
\oC(n,n)+\oC(n,x)-\oC\left(n,\frac{n}{x}\right)&=&(\log n)\bigg(\sum_{1 \le b\le x}\left\lfloor\frac{n}{b}\right\rfloor\bigg) -\lfloor x\rfloor\log\left(\left\lfloor\frac{n}{x}\right\rfloor!\right)-nH(x)\nonumber\\
&&+\lfloor x\rfloor+\sum_{1 \le b\le x}\int_1^\frac{n}{b}\frac{\{u\}}{u}\,du.\label{eqn:oCfe}
\end{eqnarray}
\end{lem}

\begin{proof}
By partial summation, we have the identity
$$
\sum_{1 \le k\le t}\log\frac{t}{k}=t-1-\int_1^t\frac{\{u\}}{u}\,du,
$$
for any $t>0$. On setting $t=\frac{n}{b}$  in this identity and summing over positive integers $b\le x$, we obtain
\begin{equation}\label{eqn:212}
\sum_{1 \le b\le x}\bigg(\sum_{1 \le k\le\frac{n}{b}}\log\frac{n}{bk}\bigg) =nH(x)-\lfloor x\rfloor-\sum_{1 \le b\le x}\int_1^\frac{n}{b}\frac{\{u\}}{u}\,du.
\end{equation}
The double sum on the left side of \eqref{eqn:212} is equal to
$$
\sum_{1 \le b\le x}\sum_{1 \le k\le\frac{n}{b}}(\log n-\log b- \log k)
=(\log n) \bigg( \sum_{1 \le b\le x}\left\lfloor\frac{n}{b}\right\rfloor\bigg) -\oC(n,x)-\sum_{1 \le k\le n}\bigg(\sum_{\substack{1 \le b\le x\\ b\le\frac{n}{k}}}\log k\bigg).
$$
On substituting the right side into \eqref{eqn:212} and rearranging, we obtain
\begin{equation}\label{eqn:213}
\sum_{1\le k\le n}\bigg(\sum_{\substack{1 \le b\le x\\b\le\frac{n}{k}}}\log k\bigg) =(\log n)\bigg(\sum_{1 \le b\le x}\left\lfloor\frac{n}{b}\right\rfloor\bigg) -\oC(n,x)-nH(x)+\lfloor x\rfloor+\sum_{1 \le b\le x}\int_1^\frac{n}{b}\frac{\{u\}}{u}\,du.
\end{equation}
The double sum on the left side of \eqref{eqn:213} is equal to
$$
\sum_{\frac{n}{x}<k\le n}\bigg(\sum_{1 \le b\le\frac{n}{k}}\log k\bigg) +\sum_{1 \le k\le\frac{n}{x}}\bigg(\sum_{1\le b\le x}\log k\bigg) =
\left(\oC(n,n)-\oC\left(n,\frac{n}{x}\right)\right)+\lfloor x\rfloor\log\left(\left\lfloor\frac{n}{x}\right\rfloor!\right).
$$
On inserting the right side into \eqref{eqn:213} and rearranging, we get \eqref{eqn:oCfe}.
\end{proof}

\begin{proof}[Proof of Proposition \ref{prop:23n}]
(1) On substituting $x=\sqrt{n}$ in Lemma \ref{lem:oCfe}, two of the terms on the left side cancel and we get
\begin{equation}\label{eq:oCnn-est} 
\oC(n,n)= (\log n)\bigg( \sum_{1 \le b\le\sqrt{n}}\left\lfloor\frac{n}{b}\right\rfloor \bigg)- \left\lfloor\sqrt{n}\right\rfloor\log\left(\left\lfloor\sqrt{n}\right\rfloor!\right)-nH\left(\sqrt{n}\right)+\left\lfloor\sqrt{n}\right\rfloor+\sum_{1 \le b\le\sqrt{n}}\int_1^\frac{n}{b}\frac{\{u\}}{u}\,du.
\end{equation}
Now, we estimate each term on the right of \eqref{eq:oCnn-est}. For the first term, we use $\lfloor t\rfloor=t+O(1)$, obtaining
\begin{eqnarray*}
(\log n) \bigg(\sum_{1 \le b\le\sqrt{n}}\left\lfloor\frac{n}{b}\right\rfloor \bigg)&=& 
(\log n) 
n H_{\lfloor\sqrt{n}\rfloor} + 
O\left(\sqrt{n}\log n\right)\\
&=& n 
(\log n) 
\left( \log \left\lfloor \sqrt{n}\right\rfloor  + \gamma +O\left(\frac{1}{\left\lfloor \sqrt{n} \right\rfloor}\right) \right) +O \left( \sqrt{n} \log n \right)\\
&=& \frac{1}{2} n (\log n)^2 +\gamma n \log n + O \left( \sqrt{n} \log n \right),
\end{eqnarray*}
where we used Lemma \ref{lem:21n} to estimate $H_{\lfloor\sqrt{n}\rfloor}$. 
For the second term, Stirling's formula gives
$$
\left\lfloor\sqrt{n}\right\rfloor\log\left(\left\lfloor\sqrt{n}\right\rfloor!\right) = \frac{1}{2} n \log n - n + O \left( \sqrt{n} \log n\right).
$$
For the third term, the harmonic number estimate in Lemma \ref{lem:21n} gives
$$
nH\left(\sqrt{n}\right)  = \frac{1}{2} n \log n + \gamma n + O \left(\sqrt{n} \right).
$$
The last two terms are negligible:
$$
\left\lfloor\sqrt{n}\right\rfloor+\sum_{1 \le b\le\sqrt{n}}\int_1^\frac{n}{b}\frac{\{u\}}{u}\,du \le\sqrt{n} + \sum_{1 \le b \le \sqrt{n} } \int_1^n\frac{1}{u}\,du  = O \left( \sqrt{n} \log n  \right). 
$$ 
Substituting these estimates into the right side of \eqref{eq:oCnn-est} yields
$$
\oC(n,n)= \frac{1}{2} n (\log n)^2 +(\gamma-1) n \log n+(1 -\gamma)n +O \left( \sqrt{n} \log n \right).
$$

(2) We will prove that for $2 \le x\le n$ 
\begin{equation}\label{eqn:2.16v32}
\oC(n,n)-\oC(n,x)=\left(H_{\left\lfloor\frac{n}{x}\right\rfloor}-\frac{x}{n}\left\lfloor\frac{n}{x}\right\rfloor\right)(n\log n-n)-\left(J_{\left\lfloor\frac{n}{x}\right\rfloor}-\frac{x}{n}\left\lfloor\frac{n}{x}\right\rfloor\log\frac{n}{x}\right)n+O\left(\frac{n\log n}{x}\right)
\end{equation}
and then deduce \eqref{eqn:oCn-integral}.\\

First, we prove \eqref{eqn:2.16v32}. We replace $x$ by $\frac{n}{x}$ in Lemma \ref{lem:oCfe}, and rearrange  a term to obtain
\begin{eqnarray}
\oC(n,n)-\oC(n,x)&=&\oC\left(n,\frac{n}{x}\right) +(\log n)\bigg(\sum_{1\le b\le\frac{n}{x}}\left\lfloor\frac{n}{b}\right\rfloor\bigg)
-\left\lfloor\frac{n}{x}\right\rfloor\log(\lfloor x\rfloor!)-nH_{\left\lfloor\frac{n}{x}\right\rfloor}\nonumber\\
&&+\bigg(\left\lfloor\frac{n}{x}\right\rfloor+\sum_{1\le b\le\frac{n}{x}}\int_1^\frac{n}{b}\frac{\{u\}}{u}\,du\bigg).\label{eqn:2.17v32}
\end{eqnarray}
We estimate the terms on the right side of  \eqref{eqn:2.17v32}.
For the first term, using $\lfloor t\rfloor=t+O(1)$, we see that
$$
\oC\left(n,\frac{n}{x}\right)=\sum_{1\le b\le\frac{n}{x}}\left\lfloor\frac{n}{b}\right\rfloor\log b=nJ_{\left\lfloor\frac{n}{x}\right\rfloor}+
O\left(\log\left(\left\lfloor\frac{n}{x}\right\rfloor!\right)+1\right).
$$
Using the bounds
$$
0 \le  \log\left(\left\lfloor\frac{n}{x}\right\rfloor!\right)\le\left\lfloor\frac{n}{x}\right\rfloor\log\left\lfloor\frac{n}{x}\right\rfloor\le\frac{n\log n}{x},
$$
 we obtain the estimate
\begin{equation}\label{eqn:2.18v32}
\oC\left(n,\frac{n}{x}\right)=nJ_{\left\lfloor\frac{n}{x}\right\rfloor}+O\left(\frac{n\log n}{x}\right).
\end{equation}
For the second term, again using $\lfloor t\rfloor=t+O(1)$  we obtain
\begin{equation}\label{eqn:2.19v32}
(\log n)\bigg( \sum_{1\le b\le\frac{n}{x}}\left\lfloor\frac{n}{b}\right\rfloor\bigg)=n(\log n)H_{\left\lfloor\frac{n}{x}\right\rfloor}+O\left(\frac{n\log n}{x}\right).
\end{equation}
For the third term we assert
\begin{equation}\label{eqn:2.20v32}
\left\lfloor\frac{n}{x}\right\rfloor\log(\lfloor x\rfloor!)=n(\log n)\frac{x}{n}\left\lfloor\frac{n}{x}\right\rfloor-n\frac{x}{n}\left\lfloor\frac{n}{x}\right\rfloor\left(1+\log\frac{n}{x}\right)+O\left(\frac{n\log n}{x}\right).
\end{equation}
This estimate follows using Stirling's formula with remainder in the form, for $x \ge 2$,
\begin{equation}\label{eqn:Stirling} 
\log\left(\lfloor x\rfloor!\right)=x\log x-x+O(\log x),
\end{equation} 
which yields 
$$
\left\lfloor\frac{n}{x}\right\rfloor\log(\lfloor x\rfloor!)=\left\lfloor\frac{n}{x}\right\rfloor(x\log x-x)+O\left(\frac{n\log x}{x}\right),
$$
and \eqref{eqn:2.20v32} follows.
For the final term we have, for $n \ge 2$ and $2 \le x\le n$, 
\begin{equation}\label{eqn:2.21v32}
\left\lfloor\frac{n}{x}\right\rfloor+\sum_{1\le b\le\frac{n}{x}}\int_1^\frac{n}{b}\frac{\{u\}}{u}\,du\le \frac{n}{x}+\sum_{1\le b\le\frac{n}{x}}\int_1^n\frac{1}{u}\,du = O \left( \frac{n\log n}{x} \right).
\end{equation}
On inserting \eqref{eqn:2.18v32}, \eqref{eqn:2.19v32}, \eqref{eqn:2.20v32}, and \eqref{eqn:2.21v32} into \eqref{eqn:2.17v32} and rearranging, we obtain \eqref{eqn:2.16v32}.\\

Next, we prove \eqref{eqn:oCn-integral}. By the substitution $v=\frac{n}{u}$, we get
\begin{equation}\label{eqn:2.22v32}
\int_x^n\left\lfloor\frac{n}{u}\right\rfloor\log u\,du=n\int_1^\frac{n}{x}\frac{\lfloor v\rfloor}{v^2}\log\frac{n}{v}\,dv.
\end{equation}
The integral has a closed form quadrature:
$$
\frac{d}{dv}\left(\frac{1}{v}-\frac{1}{v}\log\frac{n}{v}\right)=\frac{1}{v^2}\log\frac{n}{v}
$$
valid on unit intervals $b \le v < b+1$ where $\lfloor v \rfloor =b$.
By  partial summation, the right side of \eqref{eqn:2.22v32} is then equal to
$$
n\left\lfloor\frac{n}{x}\right\rfloor\left(\frac{x}{n}-\frac{x}{n}\log x\right)-n\sum_{1\le b\le\frac{n}{x}}\left(\frac{1}{b}-\frac{1}{b}\log\frac{n}{b}\right)=x(1-\log x)\left\lfloor\frac{n}{x}\right\rfloor+(n\log n-n)H_{\left\lfloor\frac{n}{x}\right\rfloor}-nJ_{\left\lfloor\frac{n}{x}\right\rfloor}.
$$
We obtain
\begin{align*}
\int_x^n\left\lfloor\frac{n}{u}\right\rfloor\log u\,du&=x(1-\log x)\left\lfloor\frac{n}{x}\right\rfloor+(n\log n-n)H_{\left\lfloor\frac{n}{x}\right\rfloor}-nJ_{\left\lfloor\frac{n}{x}\right\rfloor}\\
&=\left(H_{\left\lfloor\frac{n}{x}\right\rfloor}-\frac{x}{n}\left\lfloor\frac{n}{x}\right\rfloor\right)(n\log n-n)-\left(J_{\left\lfloor\frac{n}{x}\right\rfloor}-\frac{x}{n}\left\lfloor\frac{n}{x}\right\rfloor\log\frac{n}{x}\right)n,
\end{align*}
completing the proof.
\end{proof}

%
%

\begin{lem}\label{cor:24}
For all real numbers $n\ge2$, we have
$$
\oC\left(n,\sqrt{n}\right)=\frac{1}{8}n(\log n)^2+\gamma_1 n+O\left(\sqrt{n}\log n\right).
$$
\end{lem}

\begin{proof}
By using the estimate $\lfloor t\rfloor=t+O(1)$, we see that
\begin{equation}\label{eqn:oCsqrt}
\oC\left(n,\sqrt{n}\right)=\sum_{1\le b\le\sqrt{n}}\left\lfloor\frac{n}{b}\right\rfloor\log b=nJ\left(\sqrt{n}\right)+O\left(\log\left(\left\lfloor\sqrt{n}\right\rfloor!\right)\right).
\end{equation}
By applying Lemma \ref{lem:27} with $x=\sqrt{n}$, we obtain
$$
J\left(\sqrt{n}\right)=\frac{1}{8}(\log n)^2+\gamma_1+O\left(\frac{\log n}{\sqrt{n}}\right).
$$
Moreover, we have
$$
\log\left(\left\lfloor\sqrt{n}\right\rfloor!\right)\le\left\lfloor\sqrt{n}\right\rfloor\log\left\lfloor\sqrt{n}\right\rfloor\le\frac{1}{2}\sqrt{n}\log n.
$$
Inserting these estimates back into \eqref{eqn:oCsqrt} yields the lemma.
\end{proof}

%
%

\subsection{Estimates: $L_i(n)$}  \label{subsec:22nn4}

For positive integers $i \ge 1$ and $n \ge 2$, we set
\begin{equation} 
L_i(n) := \sum_{b=2}^n  b (\log b)^i.
\end{equation} 
We give formulas for all $i \ge 1$ but will only need
the cases $i=1,2$ in the sequel.

%
%

\begin{lem} \label{lem:Ln01-estimate} 
For all integers $i\ge1$ and $n\ge2$, we have
\begin{equation}\label{eqn:general-1} 
L_i(n)=\int_1^nu(\log u)^i\,du+\theta_i(n)n(\log n)^i,
\end{equation}
where $0\le\theta_i(n)\le1$.
 In particular,
\begin{equation} \label{eq:Ln1-estimate}
L_1(n) = \frac{1}{2} n^2 \log n - \frac{1}{4}n^2 + O\left( n \log n\right),
\end{equation} 
\begin{equation} \label{eq:Ln2-estimate}
L_2(n) = \frac{1}{2} n^2 (\log n)^2 - \frac{1}{2} n^2 \log n  + \frac{1}{4} n^2 + O\left( n (\log n)^2\right).
\end{equation} 
\end{lem} 

\begin{proof}
The function $u(\log u)^i$, $1\le u\le n$ is increasing. We have lower and upper bounds
$$
\int_1^nu(\log u)^i\,du\le\sum_{b=2}^nb(\log b)^i=L_i(n),
$$
$$
\int_1^nu(\log u)^i\,du\ge\sum_{b=1}^{n-1}b(\log b)^i=L_i(n)-n(\log n)^i.
$$
Thus the assertion \eqref{eqn:general-1}  follows.

The assertions \eqref{eq:Ln1-estimate} and \eqref{eq:Ln2-estimate}  follow from the first 
assertion with the formulas
$$
\int_1^nu\log u\,du=\left[\frac{1}{2} u^2 \log u - \frac{1}{4}u^2\right]_{u=1}^n=\frac{1}{2} n^2 \log n - \frac{1}{4}n^2+\frac{1}{4},
$$
$$
\int_1^nu(\log u)^2\,du=\left[\frac{1}{2} u^2 (\log u)^2 - \frac{1}{2} u^2 \log u  + \frac{1}{4} u^2\right]_{u=1}^n=\frac{1}{2} n^2 (\log n)^2 - \frac{1}{2} n^2 \log n  + \frac{1}{4} n^2-\frac{1}{4},
$$
completing the proof.
\end{proof} 


\begin{rem}
It can be shown by induction on $i$ that
\begin{equation}\label{eqn:general-2}
\int_1 ^n u(\log u)^i\,du=n^2\sum_{k=0}^i\frac{(-1)^kk!}{2^{k+1}}\binom{i}{k}(\log n)^{i-k}+\frac{(-1)^{i+1}i!}{2^{i+1}}.
\end{equation}
\end{rem}

%
%

\section{Estimates for $\oB(n)$}\label{sec:oB}

\par
This section and the next obtain a proof of Theorem \ref{thm:oHn} for $\oH_n$,  through  making estimates for $\oB(n)$ and $\oA(n)$ separately.

\par In this section we obtain the estimates for $\oB(n)=\sum_{b=2}^n\frac{n-1}{b-1}d_b(n)\log b$  given in Theorem \ref{thm:oBn}.

%
%

\subsection{Digit sum identity and preliminary reduction}\label{subsec:21}

\par Our estimate for  $\oB(n)$ will be  derived using the observation that  $n$ has exactly $2$ digits in base $b$ when $\sqrt{n}<b\le n$.
 
%
%

\begin{lem}\label{lem:diff1}
Let $j$ and $n$ be positive integers. Denote by $I(j,n)$ the interval $\left(\frac{n}{j+1},\frac{n}{j}\right]\cap\left(\sqrt{n},n\right]$. Then
\begin{enumerate}
\item[\emph{(1)}] $I(j,n)$ is empty unless $j<\sqrt{n}$.
\item[\emph{(2)}] If $b\in I(j,n)$ is an integer, then $d_b(n) = n- j (b-1)$, in consequence,
\begin{equation}\label{eqn:b-digitsum}
\frac{n-1}{b-1} d_b(n) \log b =  (n-1) \left( \frac{n \log b }{b-1}-  j \log b\right).  
\end{equation}
\end{enumerate}
\end{lem}

\begin{proof}
(1) Suppose that $x\in I(j,n)$. Then $\sqrt{n}<x\le\frac{n}{j}$, and hence $j<\sqrt{n}$.

(2) Since $\frac{n}{j+1}<b\le\frac{n}{j}$, it follows that $\left\lfloor\frac{n}{b}\right\rfloor=j$. Since $b>\sqrt{n}$, it follows that $\left\lfloor\frac{n}{b^i}\right\rfloor=0$ for all $i\ge2$. From \eqref{eqn:dbn-formula}, we have
$$
d_b(n)=n-(b-1)\sum_{i=1}^\infty\left\lfloor\frac{n}{b^i}\right\rfloor=n-j(b-1),
$$
and \eqref{eqn:b-digitsum} follows by multiplying by $\frac{n-1}{b-1}\log b$. This completes the proof.
\end{proof}

We split the sum $B(n)$ into three parts, 
the third part  being a  cutoff  term removing all  $2 \le b \le \sqrt{n}$,
and the first two parts  using the digit sum identity \eqref{eqn:b-digitsum}.
applied to the range $\sqrt{n} < b \le n$.

%
%

\begin{lem}\label{prop:32n}
\emph{(1)} For all integers $n \ge 2$, we have
\begin{equation}\label{eqn:identity} 
\oB(n) = \oB_1(n) - \oB_2(n) + \oB_R(n), 
\end{equation} 
in which $\oB_1(n)$, $\oB_2(n)$, and $\oB_R(n)$ are defined by
\begin{equation}
\label{eqn:oB1}
\oB_{1}(n) :=n(n-1)  \sum_{\sqrt{n} < b \leq n} \frac{\log b }{b-1}, 
\end{equation}
\begin{equation}\label{eqn:oB2}
\quad\quad\quad \,\,\,  \oB_{2}(n) := (n-1)\sum_{j=1}^{\lfloor  \sqrt{n} \,\rfloor } \, j\Bigg(\sideset{}{'}\sum_{\frac{n}{j+1}<b \leq \frac{n}{j}}  \log b\Bigg),
\end{equation}
where  the prime in the inner sum in \eqref{eqn:oB2} means only $b>\sqrt{n}$ are included, and
 \begin{equation}\label{eq:BR}
\quad \oB_R(n):=\sum_{2\le b\le\sqrt{n}}\frac{n-1}{b-1}d_b(n)\log b.
\end{equation}

\emph{(2)} For all integers $n \ge 2$, the remainder term $\oB_R(n)$ satisfies 
 \begin{equation}\label{eqn:B-remainder} 
0 \le \oB_{R}(n) \le \frac{3}{2}\,n^{3/2}\log n.
\end{equation}
\end{lem}

\begin{proof}
(1) Recall that  $\oB(n) = \sum_{b=2}^n\frac{n-1}{b-1} d_b(n) \log b.$
The remainder term $\oB_R(n)$ first cuts off the terms with $2 \le b \le \sqrt{n}$ in the sum.
The other two terms  $\oB_1(n)$ and $\oB_2(n)$ are obtained by  applying 
the decomposition \eqref{eqn:b-digitsum} of Lemma \ref{lem:diff1} to each index $b\in\left(\sqrt{n},n\right]$ term by term.  

(2) From \eqref{eqn:dbnineq}, it follows that $0\le\frac{n-1}{b-1}d_b(n)\log b\le(n-1)\log(n+1)$. Summing from $b=2$ to $\left\lfloor\sqrt{n}\right\rfloor$, we obtain
$$
0\le \oB_R(n)\le\left(\left\lfloor\sqrt{n}\right\rfloor-1\right)(n-1)\log(n+1)\le\left(\sqrt{n}\right)(n)\left(\frac{3}{2}\log n\right)=\frac{3}{2}n^{3/2}\log n
$$
as desired.
\end{proof}  

The sums $\oB_{1}(n)$ and $\oB_{2}(n)$ are  of comparable sizes, on the order of  $n^2\log n$.
We estimate them separately.

%
%

\subsection{Estimate for $\oB_{1}(n)$}
\label{subsec:B21asympt}

%
%

\begin{lem}\label{lem:oB1}
Let $\oB_{1}(n)= n(n-1)  \sum_{\sqrt{n} < b \leq n} \frac{\log b }{b-1}$.
Then for all integers $n\ge2$, we have
\begin{equation}
\label{eqn:Uncond-oB1-asymp}
\oB_{1}(n)=\frac{3}{8}n^2(\log n)^2+O\left(n^{3/2}\log n\right).
\end{equation}
\end{lem}

\begin{proof}
We rewrite the sum $\frac{\oB_{1}(n)}{n(n-1)}$ as
\begin{equation}\label{eqn:oB1n-estimate-start}
\frac{\oB_{1}(n)}{n(n-1)}=\sum_{\sqrt{n}<b\le n}\frac{\log b}{b-1}=\sum_{\sqrt{n}<b\le n}\frac{\log b}{b}+\sum_{\sqrt{n}<b\le n}\frac{\log b}{b(b-1)}.
\end{equation}
The contribution from the last sum in \eqref{eqn:oB1n-estimate-start} is negligible:
\begin{equation}\label{eqn:oB1n-estimate-1}
0\le\sum_{\sqrt{n}<b\le n}\frac{\log b}{b(b-1)}\le(\log n)\sum_{b>\sqrt{n}}\frac{1}{b(b-1)}=\frac{\log n}{\left\lfloor\sqrt{n}\right\rfloor}\le\frac{2\log n}{\sqrt{n}}.
\end{equation}
We use Lemma \ref{lem:27} to estimate the first sum on the right in \eqref{eqn:oB1n-estimate-start} and obtain
\begin{align}
\sum_{\sqrt{n}<b\le n}\frac{\log b}{b}&=J(n)-J\left(\sqrt{n}\right)\nonumber\\
&=\frac{1}{2}(\log n)^2-\frac{1}{2}\left(\log\sqrt{n}\right)^2+O\left(\frac{\log n}{\sqrt{n}}\right)\nonumber\\
&=\frac{3}{8}(\log n)^2+O\left(\frac{\log n}{\sqrt{n}}\right).\label{eqn:oB1n-estimate-2}
\end{align}
On inserting \eqref{eqn:oB1n-estimate-1} and \eqref{eqn:oB1n-estimate-2} into \eqref{eqn:oB1n-estimate-start}, we obtain
$$
\frac{\oB_{1}(n)}{n(n-1)}=\frac{3}{8}(\log n)^2+O\left(\frac{\log n}{\sqrt{n}}\right).
$$
On multiplying by $n(n-1)$, we obtain \eqref{eqn:Uncond-oB1-asymp} as desired.
\end{proof}

%
%

\subsection{Estimate for $\oB_{2}(n)$}\label{subsec:B12asympt}

%
%

\begin{lem} \label{lem:oB2}
Let 
$$\oB_{2}(n) :=  (n-1)\sum_{j=1}^{\lfloor\sqrt{n}\rfloor }j\Bigg(\sideset{}{'}\sum_{\frac{n}{j+1}<b\leq \frac{n}{j}}  \log b \Bigg),$$
where the prime in the inner sum means only $b>\sqrt{n}$ are included. Then for all integers $n \ge 2$, 
 \begin{equation} \label{eqn:oB2-asymp}
 \oB_{2}(n) =\frac{3}{8} n^{2}(\log n)^2 +(\gamma-1)n^{2}\log n +\left(1- \gamma -\gamma_1\right)n^2 +O\left(n^{3/2} \log n \right),
\end{equation}
where $\gamma$ is Euler's constant and $\gamma_1$ is the first Stieltjes constant.
\end{lem}

\begin{proof}
We have
$$
\frac{\oB_2(n)}{n-1}=\sum_{j=1}^{\lfloor\sqrt{n}\rfloor }\bigg( \sideset{}{'}\sum_{\frac{n}{j+1}<b\leq \frac{n}{j}}  \left\lfloor\frac{n}{b}\right\rfloor\log b\bigg) =\sum_{\sqrt{n}<b\le n} \left\lfloor\frac{n}{b}\right\rfloor\log b=\oC(n,n)-\oC\left(n,\sqrt{n}\right).
$$
Applying Proposition \ref{prop:23n} and Lemma \ref{cor:24} to estimate $\oC(n,n)$ and $\oC\left(n,\sqrt{n}\right)$, we obtain
\begin{align*}
\frac{\oB_2(n)}{n-1}&=\left(\frac{1}{2}n(\log n)^2+(\gamma-1)n\log n+(1-\gamma)n\right)-\left(\frac{1}{8}n(\log n)^2+\gamma_1n\right)+O\left(\sqrt{n}\log n\right)\\
&=\frac{3}{8}n(\log n)^2+(\gamma-1)n\log n+\left(1-\gamma-\gamma_1\right)n+O\left(\sqrt{n}\log n\right).
\end{align*}
On multiplying by $(n-1)$, we obtain \eqref{eqn:oB2-asymp} as desired.
\end{proof}

%
%

 \subsection{Proof of Theorem \ref{thm:oBn}}
 
 We combine the results on $\oB_1(n)$ and $\oB_2(n)$ to estimate $\oB(n)$.

\begin{proof}[Proof of Theorem \ref{thm:oBn}]
We estimate $\oB(n)$. We start from  the Lemma \ref{prop:32n} decomposition
$\oB(n)= \oB_{1}(n)- \oB_{2}(n)+\oB_R(n).$
By Lemma \ref{prop:32n} (2) we have $\oB_R(n) = O(n^{3/2}\log n)$,
which is absorbed in   the remainder term estimate  in the theorem statement.
Substituting  the formulas of Lemma \ref{lem:oB1} and Lemma \ref{lem:oB2}, we obtain
\begin{align*}
\oB(n) &= \oB_{1}(n)-\oB_{2}(n)+ \oB_{R}(n) \\
&= \frac{3}{8}n^{2}(\log n)^2 +O\left(n^{3/2} \log n\right) \\
& \quad\quad  -\left( \frac{3}{8}n^2(\log n)^2 +(\gamma-1)n^{2}\log n + (1-\gamma -\gamma_1) n^2 +O\left(n^{3/2} \log n \right)\right)  \\
&=(1-\gamma)n^{2}\log n +(\gamma+\gamma_1 -1) n^2 + O\left( n^{3/2} \log n \right),
\end{align*}
as asserted.
\end{proof}

%
%
%
%

 \section{Estimates for $\oA(n)$ and $\oH_n$}\label{sec:AH}

In this section we  derive    asymptotics for $\oA(n) = \sum_{b=2}^n \frac{2}{b-1}S_{b}(n)\log b$
given in Theorem \ref{thm:oAn} and deduce the estimate for $\log \oH_n$ given in Theorem \ref{thm:oHn}.
In the case of binomial products $\G_n$ treated in \cite{DL:22}  an asymptotic for $A(n)$ was obtained from the
relation $\log \G_n= A(n) - B(n)$, an estimate of $B(n)$,  and the existence of  a good  estimate for $\log \G_n$ coming from
its expression as a product of factorials.  Here 
we  estimate $\oA(n)$ directly. The proof details do have some parallel  with those for $\oB(n)$.

%
%

\subsection{Preliminary reduction}\label{subsec:31}

Recall that $\oA(n) = \sum_{b=2}^n\frac{2}{b-1} S_b(n) \log b$.

%
%

\begin{lem}\label{lem:31a}
For all integers $n\ge2$, we have
\begin{equation}\label{eqn:oAn-simplified} 
\oA(n) = \oA_1(n)+ O \left( n (\log n)^2 \right),
\end{equation} 
where
\begin{equation}\label{eqn:oA1n} 
\oA_1(n) := \sum_{b=2}^n \sum_{j=2}^{n}\frac{2\log b}{b}d_b(j). 
\end{equation} 
\end{lem} 

\begin{proof}
We rewrite the sum \eqref{eqn:oA-function-0}
that defines $\oA(n)$ using the identity
\begin{equation}\label{eqn:identity1}  
\frac{1}{b-1} = \frac{1}{b} + \frac{1}{b(b-1)}
\end{equation}
and obtain
\begin{equation} \label{eqn:oAn-start} 
\oA(n) = \sum_{b=2}^n \frac{2\log b}{b} S_b(n) + \sum_{b=2}^n \frac{2 \log b}{b(b-1)} S_b(n).
\end{equation}
Since $S_b(n) =\sum_{j=1}^{n-1} d_b(j)$, the first sum on the right in \eqref{eqn:oAn-start} is
\begin{equation}\label{eqn:oAn&oA1n}
\sum_{b=2}^n \frac{2\log b }{b} S_b(n) = \sum_{b=2}^n \sum_{j=1}^{n-1}\frac{2\log b}{b} d_b(j) = \oA_1(n)-\sum_{b=2}^n \frac{2\log b}{b}\left(d_b(n)-1\right).
\end{equation}
By Lemma \ref{lem:dbn-Sbn-bound},
$$
0\le d_b(n)-1\le\frac{(b-1)\log(n+1)}{\log b}-1<\frac{b\log(n+1)}{\log b}.
$$
So the last sum in \eqref{eqn:oAn&oA1n} satisfies, for all $n \ge 2,$
$$
0\le\sum_{b=2}^n \frac{2\log b}{b}\left(d_b(n)-1\right)<\sum_{b=2}^n2\log(n+1) \le 2n\log n.
$$
Hence
\begin{equation}\label{eqn:sumSestimate1}
\sum_{b=2}^n \frac{2\log b}{b}S_b(n)=\oA_1(n)+O(n\log n).
\end{equation}
Now, we treat the last sum in \eqref{eqn:oAn-start}. We apply Lemma \ref{lem:dbn-Sbn-bound}  to bound $S_b(n)$, obtaining
\begin{equation}\label{eqn:sumSestimate2}
\sum_{b=2}^n \frac{2 \log b}{b(b-1)} S_b(n)\le\sum_{b=2}^n\frac{n\log n}{b}\ll n(\log n)^2.
\end{equation}
On inserting \eqref{eqn:sumSestimate1} and \eqref{eqn:sumSestimate2} into \eqref{eqn:oAn-start}, we obtain \eqref{eqn:oAn-simplified} as desired.
\end{proof}

%
%

\subsection{Estimate for $\oA_{1}(n)$ reduction}\label{subsec:32}
\label{sec:oA1asympt}

%
%

\begin{lem} \label{lem:32a}
\emph{(1)} For all integers $n \ge 2$, the sum $\oA_1(n)$ given by \eqref{eqn:oA1n} can be rewritten as
\begin{equation}\label{eqn:oAn-simplified2} 
\oA_1(n) = \oA_{11}(n) + \oA_{12}(n) - \oA_{R}(n),
\end{equation}
where
\begin{equation}\label{eqn:oA11n} 
\oA_{11}(n) := \sum_{j=2}^n \frac{2}{j-1} \oB(j),
\end{equation}
\begin{equation}\label{eqn:oA12n} 
\quad\quad \oA_{12} (n) := \sum_{j=2}^n \sum_{b=j+1}^n \frac{2j\log b}{b},
\end{equation}
\begin{equation}\label{eqn:oARn} 
\quad\quad\quad \oA_{R} (n) :=  \sum_{j=2}^{n} \sum_{b=2}^j \frac{2\log b}{b(b-1)} d_b(j),
\end{equation}
and $\oB(n)$ is given by \eqref{eqn:oB-function-0}.\\

\emph{(2)} For all integers $n \ge 2$, we have
\begin{equation} 
\oA_R(n) \le 3 n (\log n)^2.
\end{equation}
\end{lem} 

\begin{proof}
(1) 
We start from \eqref{eqn:oA1n} and interchange the order of summation, obtaining
\begin{eqnarray}\label{eq:410} 
\oA_1(n) &=& \sum_{b=2}^n \sum_{j=2}^{n} \frac{2\log b}{b}  d_b(j) \nonumber \\
& =  &  \sum_{j=2}^{n} \sum_{b=2}^n  \frac{2\log b}{b}d_b(j)  \nonumber\\
&=& \sum_{j=2}^{n} \sum_{b=2}^j \frac{2\log b}{b}d_b(j)+ \sum_{j=2}^{n} \sum_{b=j+1}^n \frac{2\log b}{b}d_b(j).
 \end{eqnarray}
Recall that $\oB(j) =\sum_{b=2}^j \frac{j-1}{b-1}d_{b}(j)\log b$.  We next use the identity \eqref{eqn:identity1}  
to rewrite the first sum on the right in \eqref{eq:410}:
$$
 \sum_{j=2}^{n} \sum_{b=2}^j \frac{2\log b}{b} d_b(j) = \sum_{j=2}^{n}\sum_{b=2}^j\frac{2\log b}{b-1}d_b(j)-\sum_{j=2}^{n}\sum_{b=2}^j \frac{2\log b}{b(b-1)} d_b(j) = \oA_{11}(n)  - \oA_{R}(n). 
$$
Finally, we note that $d_b(j)=j$ for $j<b$; so the second sum on the right in \eqref{eq:410} is
$$
\sum_{j=2}^{n}\sum_{b=j+1}^n \frac{2\log b}{b}d_b(j) = \sum_{j=2}^{n} \sum_{b=j+1}^n  \frac{2j\log b}{b}=\oA_{12}(n).
$$

(2) We first bound $\oA_{R}(n)$ by 
$$
0 \le  \oA_{R}(n) \le \sum_{j=2}^n\sum_{b=2}^{n} \frac{2\log b}{b(b-1)} d_b(j) = \sum_{b=2}^n \frac{2 \log b}{b(b-1)} \left(S_b(n+1)-1\right).
$$ 
Applying Lemma \ref{lem:dbn-Sbn-bound}, to bound the last quantity, we obtain for $n \ge 2$, 
$$
\oA_{R}(n) < \sum_{b=2}^n \frac{(n+1)\log(n+1)}{b}  \le  3n(\log n)^2,
$$
as asserted.
\end{proof}

%
%

\subsection{Estimates for $\oA_{11}(n)$ and $\oA_{12}(n)$}\label{sec:33}
\label{subsec:oA11asympt}

%
%

\begin{lem} \label{lem:oA11}

For all integers $n \ge 2,$ we have
\begin{equation}\label{eqn:oA11bound} 
\oA_{11}(n)=(1-\gamma)n^2\log n+\left(\frac{3}{2}\gamma+\gamma_1-\frac{3}{2}\right)n^2+O\left(n^{3/2}\log n\right).
\end{equation} 
\end{lem}

\begin{proof}
We start from \eqref{eqn:oA11n} and use the identity \eqref{eqn:identity1} to rewrite $\oA_{11}(n)$:
\begin{equation}\label{eqn:A11n-start}
\oA_{11}(n)=\sum_{j=2}^{n} \frac{2}{j-1}\oB(j)=\sum_{j=2}^n\frac{2}{j}\oB(j)+\sum_{j=2}^n\frac{2}{j(j-1)}\oB(j).
\end{equation}
From Theorem \ref{thm:oBn}, it follows that $\oB(j)\ll j(j-1)\log j$ for $j\ge2$. As a result, the contribution from the last sum in \eqref{eqn:A11n-start} is negligible:
\begin{equation}\label{eqn:oA11n-part1}
\sum_{j=2}^n\frac{2}{j(j-1)}\oB(j)\ll \sum_{j=2}^n\log j\le\sum_{j=2}^n\log n<n\log n.
\end{equation}
Now, we estimate the first sum on the right of \eqref{eqn:A11n-start} using Theorem \ref{thm:oBn}:
\begin{align*}
\sum_{j=2}^n\frac{2}{j}\oB(j)&=2(1-\gamma)\sum_{j=2}^nj\log j+2\left(\gamma+\gamma_1-1\right)\sum_{j=2}^nj+O\bigg(\sum_{j=2}^n\sqrt{j}\log j\bigg)\\
&=2(1-\gamma)\sum_{j=2}^nj\log j+2\left(\gamma+\gamma_1-1\right)\left(\frac{1}{2}n^2+\frac{1}{2}n-1\right)+O\bigg(\sum_{j=2}^n\sqrt{n}\log n\bigg)\\
&=2(1-\gamma)\sum_{j=2}^nj\log j+\left(\gamma+\gamma_1-1\right)n^2+O\left(n^{3/2}\log n\right).
\end{align*}
We use Lemma \ref{lem:Ln01-estimate} to estimate $\sum_{j=2}^nj\log j$ and obtain
\begin{equation}\label{eqn:oA11n-part2}
\sum_{j=2}^n\frac{2}{j}\oB(j)=(1-\gamma)n^2\log n+\left(\frac{3}{2}\gamma+\gamma_1-\frac{3}{2}\right)n^2+O\left(n^{3/2}\log n\right).
\end{equation}
On inserting \eqref{eqn:oA11n-part1} and \eqref{eqn:oA11n-part2} into \eqref{eqn:A11n-start}, we obtain \eqref{eqn:oA11bound} as desired.
\end{proof}

%
%

\begin{lem} \label{lem:oA12}
For all integers $n \ge 2$, we have
$$\oA_{12}(n)=\frac{1}{2}n^2\log n-\frac{1}{4}n^2+O\left(n(\log n)^2\right).$$
\end{lem}

\begin{proof} We can rewrite \eqref{eqn:oA12n} in terms of $J(x)=\sum_{1\le b\le x}\frac{\log b}{b}$ as
$$\oA_{12} (n) = \sum_{j=2}^n 2j (J(n) - J( j)).$$
For $2\le j\le n$, it follows from Lemma \ref{lem:27} that
$$
J(n)- J(j) = \frac{1}{2} (\log n)^2- \frac{1}{2}(\log j)^2  + O \left(\frac{\log n}{j}\right).
$$
Hence
\begin{align*}
\oA_{12}(n) &=\sum_{j=2}^n j (\log n)^2- \sum_{j=2}^n j(\log j)^2 +O\bigg(\sum_{b=2}^n\log n\bigg)\\
&= \left(\frac{1}{2}n^2+\frac{1}{2}n -1\right)(\log n)^2 -  \sum_{j=2}^n j (\log j)^2 + O \left( n \log n \right).
\end{align*} 
We use Lemma \ref{lem:Ln01-estimate} to estimate $\sum_{j=2}^nj(\log j)^2$ and obtain
$$
\oA_{12}(n) = \frac{1}{2} n^2 \log n -\frac{1}{4} n^2 + O \left( n (\log n)^2 \right),
$$
as desired.
\end{proof}

%
%

 \subsection{Proof of Theorem \ref{thm:oAn}  }\label{sec:35}

\begin{proof}[Proof of Theorem \ref{thm:oAn}]
By Lemma \ref{lem:31a} and Lemma \ref{lem:32a},
$$
\oA(n)  = \oA_1(n)+O \left( n (\log n)^2\right)= \oA_{11}(n) + \oA_{12}(n) + O \left( n (\log n)^2\right). 
$$
Inserting the estimates of Lemma \ref{lem:oA11} for $\oA_{11}(n)$ and Lemma \ref{lem:oA12} for $\oA_{12}(n)$ yields
$$
\oA (n) = \left(\frac{3}{2} -\gamma\right)n^2\log n+\left(\frac{3}{2}\gamma+\gamma_1-\frac{7}{4}\right)n^2+O\left(n^{3/2}\log n\right), 
$$
as required.
\end{proof}
%
%

 \subsection{Proof of Theorem  \ref{thm:oHn} }\label{sec:36}

\begin{proof}[Proof of Theorem \ref{thm:oHn}]
The estimate for $\oH_n$  follows  from the relation $\log \oH_n = \oA(n) - \oB(n)$ using
the estimates of Theorem \ref{thm:oBn} for $\oB(n)$ and Theorem  \ref{thm:oAn} for $\oA(n)$.
\end{proof}

%
%

\section{Estimates for the generalized partial factorization sums $\oB(n,x)$}\label{sec:asymp-oBnx}

This section and the next obtain a proof of Theorem \ref{thm:oHnx-main}, 
through making  estimates for $\oB(n,x)$ and $\oA(n,x)$ separately.

We derive estimates for $\oB(n,x)$ in the interval $1 \le x  \le n$
by varying $x$ downward, starting from the asymptotic estimates for $\oB(n)= \oB(n,n)$ in Theorem \ref{thm:oBn},
 useable as a black box.

In the next result,   $H_m= \sum_{k=1}^m \frac{1}{k}$ denotes the $m$-th harmonic number.

%
%

\begin{thm}\label{thm:oBnx}
Let $\oB(n,x) = \sum_{b=2}^{\lfloor x\rfloor}\frac{n-1}{b-1} d_b(n) \log b.$ 
Then for all integers $n \ge 2$ and real $x\in\left[\sqrt{n},n\right]$, 
\begin{equation}\label{eqn:oBnx-bound}
\oB(n,x) = \oB_0(n,x) \, n^2 \log n + \oB_1(n, x) \, n^2 + O \left( n^{3/2} \log n\right),
\end{equation}
where the functions $\oB_0(n,x)$ and $\oB_1(n,x)$ only depend on $\frac{x}{n}$ and are given by
\begin{equation}\label{eqn:oB0nx-bound}
\oB_0(n,x)  := \left(1-\gamma\right)+ \left(  H_{\lfloor \frac{n}{x} \rfloor} -  \log \frac{n}{x}\right)
- \frac{x}{n}\left\lfloor \frac{n}{x} \right\rfloor 
\end{equation}
and 
\begin{align}
\oB_1(n,x) := & 
 \left(\gamma+ \gamma_1 -1\right)- \left( H_{\lfloor \frac{n}{x}\rfloor} - \log \frac{n}{x}   \right)  - \left( J_{\lfloor \frac{n}{x}  \rfloor} - \frac{1}{2}\left(\log \frac{n}{x}  \right)^2 \right)\nonumber\\
& - \log \frac{n}{x}+ \left(\log \frac{n}{x}\right) \frac{x}{n} \left\lfloor \frac{n}{x} \right\rfloor  +  \frac{x}{n}\left\lfloor \frac{n}{x}  \right\rfloor.\label{eqn:oB1nx-bound}
\end{align}
Moreover, for all integers $n\ge2$ and real $x\in\left[1,\sqrt{n}\right]$, 
\begin{equation}\label{eqn:oBnx-bound2} 
\oB(n, x) = O \left( n^{3/2} \log n \right) .
\end{equation} 
\end{thm} 


\begin{rem}\label{rmk:52}
The answer displays a {\em scale invariance}, in terms of  the variables $x$
and $n$.  That is, the  functions $\oB_0(n,x)$ and $\oB_1(n, x)$ above turn out to be  functions of a single variable $\alpha :=\frac{x}{n}$ 
having $0 \le \alpha \le 1.$  However in the proof itself, various calculations contain   terms in $n$ and $x$ that are {\em not} scale invariant.
\end{rem}

%
%

\subsection{Preliminary reduction} \label{subsec:prelim5}

We write
\begin{equation}
\oB(n, x) = \oB(n) - \oB^c(n,x),
\end{equation} 
where $ \oB^c(n,x)$ is the  complement sum
\begin{equation}\label{eqn:Bcnx}
\oB^c(n,x) := 
\sum_{x < b \le n} \frac{n-1}{b-1} d_b(n) \log b.
\end{equation}
The sum $\oB(n)$ can be estimated by Theorem \ref{thm:oBn}.
To estimate $\oB^c(n,x)$, we break it into two parts.

%
%

\begin{lem}\label{lem:51a}
For all integers $n \ge 2$ and real numbers $x$ such that $\sqrt{n} \le x\le n$, we have
\begin{equation}\label{eqn:Bcc}
\oB^c(n,x) = \oB_{11}^c(n,x) - (n-1) \left(\oC(n, n) - \oC(n,x)\right),
\end{equation} 
where $\oC(n,x)$ is given by \eqref{eqn:oC-defn} and
\begin{equation}\label{eqn:B11c} 
\oB_{11}^c(n,x) := n (n-1) \sum_{x < b \le n} \frac{\log b}{b-1}.
\end{equation}
\end{lem}

\begin{proof}
Recall from \eqref{eqn:dbn-formula} that
$d_b(n)= n- (b-1)\sum_{i=1}^{\infty} \left\lfloor \frac{n}{b^i} \right\rfloor$.
Since $x > \sqrt{n}$, if  $b>x$, then $b^2>x^2\ge n$, and hence $\left\lfloor\frac{n}{b^i}\right\rfloor=0$ for all $i\ge2$. In this circumstance, $d_b(n)=n-(b-1)\left\lfloor\frac{n}{b}\right\rfloor$ for $b>x$. 
Inserting this formula into  the definition \eqref{eqn:Bcnx}, we obtain
\begin{align}
\oB^c(n,x)&=n(n-1)\sum_{x<b\le n}\frac{\log b}{b-1}-(n-1)\sum_{x<b\le n}\left\lfloor\frac{n}{b}\right\rfloor\log b\nonumber\\
&= \oB_{11}^c(n,x) - (n-1) \left(\oC(n, n) - \oC(n,x)\right)
\end{align}
as required.
\end{proof}

%
%

\subsection{Estimate for $\oB_{11}^c(n,x)$}\label{sec:51}

%
%

\begin{lem}\label{lem:52}
For all real numbers $n \ge 2$ and $x$ such that $1\le x\le n$, we have
\begin{equation}\label{eqn:oBc11-estimate}
\oB_{11}^c(n,x) = \frac{1}{2}n^2(\log n)^2-\frac{1}{2}n^2(\log x)^2+O\left(\frac{n^2\log n}{x}\right).
\end{equation}
\end{lem}

\begin{proof}
We start from \eqref{eqn:B11c} and use the identity \eqref{eqn:identity1} to rewrite $\frac{1}{n(n-1)}\oB_{11}^c(n, x)$:
\begin{equation}\label{eqn:oBc11-start}
\frac{1}{n(n-1)}\oB_{11}^c(n, x) = \sum_{x < b \le n} \frac{\log b}{b} + \sum_{x< b \le n} \frac{\log b}{b (b-1)}.
\end{equation}
The contribution from the last sum in \eqref{eqn:oBc11-start} is negligible:
\begin{equation}\label{eqn:oBc11-error}
 0\le \sum_{x< b \le n} \frac{\log b}{b (b-1)}<(\log n)\sum_{b>x}\frac{1}{b(b-1)}=\frac{\log n}{\lfloor x\rfloor}<\frac{2\log n}{x}.
\end{equation}
The first sum on the right in \eqref{eqn:oBc11-start} can be estimated using Lemma \ref{lem:27}: 
\begin{equation}\label{eqn:oBc11-main}
\sum_{x < b \le n} \frac{\log b}{b}=J(n)- J(x)=\frac{1}{2}(\log n)^2 - \frac{1}{2}(\log x)^2  + O\left( \frac{\log n}{x}\right).
\end{equation}
On inserting \eqref{eqn:oBc11-error} and \eqref{eqn:oBc11-main} into \eqref{eqn:oBc11-start}, we obtain
$$
\frac{1}{n(n-1)}\oB_{11}^c(n,x) = \frac{1}{2}(\log n)^2-\frac{1}{2}(\log x)^2+O\left(\frac{\log n}{x}\right).
$$
On multiplying by $n(n-1)$, we obtain
$$
\oB_{11}^c(n,x) = \frac{1}{2}n^2(\log n)^2-\frac{1}{2}n^2(\log x)^2-\left(\frac{1}{2}n(\log n)^2-\frac{1}{2}n(\log x)^2\right)+O\left(\frac{n^2\log n}{x}\right).
$$
Since $e^t=1+t+\frac{t^2}{2}+\dots>t$ for $t>0$, it follows that $\log\frac{n}{x}<\frac{n}{x}$ and
$$
0\le\frac{1}{2}n(\log n)^2-\frac{1}{2}n(\log x)^2=\frac{1}{2}n\left(\log\frac{n}{x}\right)(\log n+\log x)<\frac{1}{2}n\left(\frac{n}{x}\right)(2\log n)=\frac{n^2\log n}{x}.
$$
Hence \eqref{eqn:oBc11-estimate} follows.
\end{proof}

%
%

\subsection{Estimate for $\oB^c(n,x)$}\label{subsec:54}

We obtain an asymptotic estimate for $\oB^c(n,x)$.

%
%

\begin{prop}\label{prop:56} 
For all integers $n \ge 2$ and real numbers $x$ such that $\sqrt{n}\le x\le n$, we have
\begin{eqnarray}\label{eqn:oBc0} 
 \oB^c(n,x) &=& n^2\left( J \left( \frac{n}{x} \right) - \frac{1}{2} \left(\log \frac{n}{x}\right)^2 \right)  -
 (n^2 \log n - n^2) \left( H_{\lfloor \frac{n}{x} \rfloor}  - \log \frac{n}{x}\right) \nonumber \\
  &&+\,n^2\left( 1- \left\lfloor \frac{n}{x} \right\rfloor \frac{x}{n} \right) \log \frac{n}{x} 
 + (n^2\log n -n^2) \left\lfloor \frac{n}{x} \right\rfloor \frac{x}{n}
 +O \left( n^{3/2} \log n\right).
 \end{eqnarray} 
\end{prop}
\begin{proof}
We have
\begin{equation}\label{eqn:oBnc}
\oB^c(n,x)  = \oB_{11}^c(n,x) - (n-1)\left( \oC(n,n)- \oC(n, x)\right). 
\end{equation} 
We suppose $\sqrt{n} \le x\le n.$ From  Lemma \ref{lem:52} we obtain
\begin{equation}\label{eqn:oB11c} 
\oB_{11}^c(n,x) = \frac{1}{2} n^2 \left( (\log n)^2 - (\log x)^2\right) + O \left( n^{3/2} \log n\right).
\end{equation} 
Now Proposition  \ref{prop:23n}(2) gives for $2 \le x \le n$, 
\begin{equation}
\oC(n,n)-\oC(n,x) =
\left(H_{\left\lfloor\frac{n}{x}\right\rfloor}-\frac{x}{n}\left\lfloor\frac{n}{x}\right\rfloor\right)(n\log n-n) 
- \left( J_{\left\lfloor\frac{n}{x}\right\rfloor}- \frac{x}{n}\left\lfloor\frac{n}{x}\right\rfloor \log\frac{n}{x}
 \right)n +O\left(\frac{n\log n}{x}\right).
\end{equation} 
Substituting these estimates into \eqref{eqn:oBnc}, and assuming $x \ge \sqrt{n}$  yields
\begin{eqnarray}\label{eq:version1}
\oB^c(n,x)  &= &\frac{1}{2} n^2 \left( (\log n)^2 - (\log x)^2\right) + n \left(H_{\left\lfloor\frac{n}{x}\right\rfloor} 
-\frac{x}{n}\left\lfloor\frac{n}{x}\right\rfloor  \right)(n\log n-n)\nonumber \\
&&- n \left( J_{\left\lfloor\frac{n}{x}\right\rfloor}- \frac{x}{n}\left\lfloor\frac{n}{x}\right\rfloor \log\frac{n}{x}
\right) + O \left( n^{3/2}( \log n)^2 \right). 
\end{eqnarray} 
In this formula we also replaced  a factor $(n-1)$ by $n$, introducing an error $O ( n (\log n)^2)$ absorbed in the remainder term. 
Our goal is to simplify this expression to obtain \eqref{eqn:oBc0}. We  rewrite \eqref{eq:version1} as
\begin{eqnarray}
\oB^c(n,x)  &=&      n^2 \left( J_{\left\lfloor\frac{n}{x}\right\rfloor} - \frac{1}{2}\left(\log \frac{n}{x}\right)^2 \right )
+(n^2 \log n -n^2 )\left( H_{\left\lfloor\frac{n}{x}\right\rfloor} - \log \frac{n}{x} \right)\nonumber\\
&&+ \oB_{21}^c(n,x)+O \left( n^{3/2} (\log n)^2\right),\label{eq:version2}
\end{eqnarray}
where we define
\begin{eqnarray} \label{eq:version3}
\oB_{21}^c(n,x) &:= &\left( \frac{1}{ 2}n^2 \left(\log \frac{n}{x}\right)^2 - n^2 (\log n) \left( \log \frac{n}{x}\right)  + n^2 \left(\log \frac{n}{x}\right) \right)  +\frac{1}{2} n^2 \left( (\log n)^2 - (\log x)^2\right)\nonumber \\
&&
-\frac{x}{n}\left\lfloor\frac{n}{x}\right\rfloor \left(n^2\log n -n^2\right)
- \left(\frac{x}{n}\left\lfloor\frac{n}{x}\right\rfloor\right) n^2\log\frac{n}{x} .
\end{eqnarray} 
Expanding  $\log \frac{n}{x} = \log n - \log x$  in  the first two terms in \eqref{eq:version3}  gives 
$$
\frac{1}{2} \left(\log \frac{n}{x}\right)^2- n^2 (\log n) \left(\log \frac{n}{x}\right) = -\frac{1}{2} n^2(\log n)^2 + \frac{1}{2} n^2 (\log x)^2,
$$
which cancels the next two terms appearing in \eqref{eq:version3}.
Rearranging the remaining uncancelled terms results in 
\begin{equation} \label{eq;version4}
\oB_{21}^c(n,x) = \frac{x}{n}\left\lfloor\frac{n}{x}\right\rfloor \left(n^2\log n -n^2\right)
+ \left(1- \frac{x}{n}\left\lfloor\frac{n}{x}\right\rfloor\right) n^2 \log\frac{n}{x},
\end{equation} 
which  when substituted in \eqref{eq:version2} yields \eqref{eqn:oBc0}. 
\end{proof}

%
%

\subsection{Proof of Theorem \ref{thm:oBnx}}\label{subsec:55}

We obtain an estimate of $\oB(n,x)$.

\begin{proof}[Proof of Theorem \ref{thm:oBnx}]
For $n \ge x \ge 1$ we  have the decomposition
\begin{equation}
\oB(n,x) = B(n) -\oB^c(n,x).
\end{equation}
By Theorem \ref{thm:oBn} we have for $n \ge 2$, 
$$
 \oB(n)=(1-\gamma)n^2\log n+\left(\gamma+\gamma_1-1\right)n^2+O\left(n^{3/2}\log n\right).
$$
By Proposition \ref{prop:56} we have for $n \ge 2$ and $\sqrt{n} \le x \le n$, 
\begin{eqnarray*} 
\oB^c(n,x)  &:=&  n^2\left( J_{\lfloor \frac{n}{x} \rfloor} - \frac{1}{2} \bigg(\log \frac{n}{x}\bigg)^2 \right)  -
 \left(n^2 \log n - n^2\right) \left( H_{\lfloor \frac{n}{x} \rfloor}  - \log \frac{n}{x} \right) \nonumber \\
  && + \, n^2\left( 1- \bigg\lfloor \frac{n}{x} \bigg\rfloor \frac{x}{n} \right) \log \frac{n}{x} 
 + \left(n^2\log n -n^2\right)\bigg\lfloor \frac{n}{x} \bigg\rfloor \frac{x}{n} 
 +O \left( n^{3/2} \log n\right). 
\end{eqnarray*}
We obtain for $n \ge 2$ and $\sqrt{n} \le x \le n$, 
\begin{eqnarray}\label{eqn:oBnx-estimate51} 
\oB(n,x) & = & \oB_0(n, x) \, n^2 \log n + \oB_1 (n,x) \, n^2 + O \left(n^{3/2}\log n \right),
\end{eqnarray}
with
\begin{eqnarray*}
 \oB_0(n, x) &= & (1-\gamma) +\left( H_{\lfloor \frac{n}{x} \rfloor}  - \log \frac{n}{x} \right) -\bigg\lfloor \frac{n}{x} \bigg\rfloor \frac{x}{n}\quad\quad\quad\quad
\end{eqnarray*} 
and
\begin{eqnarray*}
 \oB_1(n, x) &=& \left(\gamma+\gamma_1-1\right)- 
 \left( H_{\lfloor \frac{n}{x} \rfloor}  - \log \frac{n}{x} \right)- \left( J_{ \lfloor \frac{n}{x} \rfloor} - \frac{1}{2} \bigg(\log \frac{n}{x}\bigg)^2 \right) 
 \\&&
  - \left( 1- \bigg\lfloor \frac{n}{x} \bigg\rfloor \frac{x}{n} \right)\log \frac{n}{x}   +\bigg\lfloor \frac{n}{x} \bigg\rfloor \frac{x}{n},
\end{eqnarray*} 
which is \eqref{eqn:oBnx-bound}.

Finally, for integers $n\ge2$ and real $x\in\left[1,\sqrt{n}\right]$, we have
\begin{align*}
\oB(n,x)&=\sum_{2\le b\le x}\frac{n-1}{b-1}d_b(n)\log b\\
&\le\sum_{2\le b\le x}(n-1)\log(n+1)\\
&<x(n-1)\log(n+1)\\
&<2n^{3/2}\log n,
\end{align*} 
where the bound of Lemma \ref{lem:dbn-Sbn-bound} for $d_b(n)$ was used in the first inequality.
We have obtained \eqref{eqn:oBnx-bound2}.
\end{proof}

%
%

\subsection{Proof of Theorem \ref{thm:oBnx-cor}}\label{subsec:56}

\begin{proof}[Proof of Theorem \ref{thm:oBnx-cor}]
We estimate $\oB(n, \alpha n)$.
The theorem follows on choosing  $x= \alpha n$ in Theorem \ref{thm:oBnx} and simplifying. Now $\oB_0(n, x)= f_{\oB}(\alpha)  $ is
a function of $\alpha= \frac{x}{n}$, with 
$$
f_{\oB}(\alpha) = (1-\gamma) + \left( H_{\lfloor \frac{1}{\alpha}\rfloor} - \log \frac{1}{\alpha} \right) - \alpha \bigg\lfloor \frac{1}{\alpha} \bigg\rfloor.
$$
Similarly $\oB_1(n,x) = g_{\oB}(\alpha)$ is a function of $\alpha$ with 
\begin{eqnarray*}
g_{\oB}(\alpha) &=& \left(\gamma+\gamma_1-1 \right)- 
 \left( H_{\lfloor \frac{1}{\alpha}  \rfloor}  - \log \frac{1}{\alpha} \right) - \left( J_{ \lfloor \frac{1}{\alpha} \rfloor} - 
 \frac{1}{2} \bigg(\log \frac{1}{\alpha}\bigg)^2 \right)  \\
 &&\quad + \left( \alpha \bigg\lfloor \frac{1}{\alpha}  \bigg\rfloor -1 \right) \log \frac{1}{\alpha}   + \bigg\lfloor \frac{1}{\alpha} \bigg\rfloor \alpha.
\end{eqnarray*} 
We allow $\frac{1}{\sqrt{n}} \le \alpha \le 1$, and for $n \ge 2$ the remainder term in the estimate is $O \left( n^{3/2} \log n\right)$,
independent of $\alpha$ in this range. For the range $x\in \left[1, \sqrt{n}\right]$ we use the
final estimate \eqref{eqn:oBnx-bound2}.
\end{proof}


\begin{rem}\label{rem:34}
The function $f_{\oB}(\alpha)$ has  $f_{\oB}(1)=1-\gamma$, and has  $\lim_{\alpha \to 0^+} f_{\oB}(\alpha)=0$   since $ H_{\lfloor \frac{1}{\alpha}\rfloor}- \log \frac{1}{\alpha}  \to \gamma$
as $\alpha \to 0^+$.  Various individual terms in the formulas for $f_{\oB}(\alpha)$ and $g_{\oB}(\alpha)$ are discontinuous at points $\alpha=\frac{1}{k}$ for $k \ge 1$.
The function $f_{\oB}(\alpha)$ was shown to be continuous on $[0,1]$ in \cite{DL:22}; the function $g_{\oB}(\alpha)$ can be checked to be continuous. 
\end{rem}

%
%

\section{Estimates for the generalized partial factorization sums $\oA(n,x)$}\label{sec:asymp-oAnx}

The main goal of this section is to prove the following theorem.

%
%

\begin{thm}\label{thm:oAnx}
Let $\oA(n,x) = \sum_{b=2}^{\lfloor x\rfloor}\frac{2}{b-1} S_b(n) \log b.$ 
Then for all integers $n \ge 2$ and real $x\in\left[\sqrt{n},n\right]$, 
\begin{equation}\label{eqn:oAnx-bound}
\oA(n,x) =  \oA_0(n,x) n^2 \log n + \oA_1(n, x) n^2 + O \left(n^{3/2} \log n \right),
\end{equation}
where the functions $\oA_0(n,x)$ and $\oA_1(n,x)$ only depend on $\frac{x}{n}$ and are given by
\begin{equation}\label{eqn:oA0-thm}
\oA_0(n, x) := \left(\frac{3}{2} - \gamma\right) + \left(H_{\lfloor \frac{n}{x} \rfloor} - \log \frac{n}{x} \right) 
+\frac{1}{2} \left(\frac{x}{n}\right)^2 \left\lfloor \frac{n}{x} \right\rfloor \left\lfloor \frac{n}{x} + 1\right\rfloor - 2 \left(\frac{x}{n}\right) \left\lfloor \frac{n}{x} \right\rfloor
\end{equation}
and
\begin{align}\label{eqn:oA1-thm} 
\oA_1(n, x) := &\left(\frac{3}{2} \gamma + \gamma_1 -\frac{7}{4}\right)- \frac{3}{2} \left(H_{\lfloor \frac{n}{x} \rfloor} - \log \frac{n}{x} \right) - \left( J_{\lfloor \frac{n}{x} \rfloor}  - \frac{1}{2} \left(\log \frac{n}{x}\right)^2 \right)  - \frac{3}{2} \log \frac{n}{x}  \nonumber   \\
&  - \frac{1}{2} \left(\log \frac{n}{x}\right) \left(\frac{x}{n}\right)^2 \left\lfloor \frac{n}{x} \right\rfloor \left\lfloor \frac{n}{x} +1\right\rfloor + 2 \left(\log \frac{n}{x}\right)  \frac{x}{n} \left\lfloor \frac{n}{x} \right\rfloor\nonumber\\
&-\frac{1}{4}  \left(\frac{x}{n}\right)^2 \left\lfloor \frac{n}{x} \right\rfloor \left\lfloor \frac{n}{x}+1 \right\rfloor +
2\left(\frac{x}{n}\right)  \left\lfloor \frac{n}{x} \right\rfloor.
\end{align}
Moreover, for all integers $n\ge2$ and real $x\in\left[1,\sqrt{n}\right]$,
\begin{equation}\label{eqn:oAnx-bound2}
\oA(n,x)=O\left(n^{3/2}\log n\right).
\end{equation}
\end{thm}

We derive estimates for $\oA(n,x)$ starting from $\oA(n,n)$ and working downward, via a recursion in Lemma \ref{lem:oA-diff-formula} below.

%
%

\subsection{Estimates for the complement sum $\oA(n,n)-\oA(n,x)$}

First, we show that $\oA(n,n)-\oA(n,x)$ can be written in terms of known quantities, namely $\oB^c_{11}(n,x)$ and $\oC(j,j)-\oC(j,x)$.

%
%

\begin{lem}\label{lem:oA-diff-formula}
For all integers $n\ge2$ and real numbers $x$ such that $\sqrt{n-1}\le x\le n$, we have
\begin{equation}\label{eqn:oAnn-oAnx-formula}
\oA(n,n)-\oA(n,x)=\oB^c_{11}(n,x)-2\sum_{x\le j<n}\left(\oC(j,j)-\oC(j,x)\right),
\end{equation}
where $\oC(n,x)$ and $\oB^c_{11}(n,x)$ are given by \eqref{eqn:oC-defn} and \eqref{eqn:B11c}, respectively.
\end{lem}

\begin{proof}
From \eqref{eqn:oA-function}, we have
\begin{equation}\label{eqn:oAnn-oAnx-start}
\oA(n,n) =\oA(n,x) + \sum_{x<b\le n}\frac{2}{b-1}S_b(n)\log b  
\end{equation}
Observe that for positive integers $b>x$ and $j\le n-1$, we have $b^2>x^2\ge n-1\ge j$, and hence $\left\lfloor \frac{j}{b^i}\right\rfloor=0$ for all $i\ge2$. From \eqref{eqn:Sbn-formula}, if $b>x$, then
$$
S_b(n)=\frac{n(n-1)}{2}-(b-1)\sum_{j=1}^{n-1}\left\lfloor\frac{j}{b}\right\rfloor.
$$
On inserting this into \eqref{eqn:oAnn-oAnx-start}, we obtain
\begin{align*}
\oA(n,n)-\oA(n,x)&=n(n-1)\sum_{x<b\le n}\frac{\log b}{b-1}-2\sum_{j=1}^{n-1}\sum_{x<b\le n}\left\lfloor\frac{j}{b}\right\rfloor\log b\\
&=\oB^c_{11}(n,x)-2\sum_{j=1}^{n-1}\left(\oC(j,n)-\oC(j,x)\right).
\end{align*}
From \eqref{eqn:oCstab}, if $1\le j<n$, then $\oC(j,n)=\oC(j,j)$. Hence
$$
\oA(n,n)-\oA(n,x)=\oB^c_{11}(n,x)-2\sum_{1\le j<n}\left(\oC(j,j)-\oC(j,x)\right).
$$
From \eqref{eqn:oCstab}, if $1\le j<x$, then $\oC(j,x)=\oC(j,j)$. Hence \eqref{eqn:oAnn-oAnx-formula} follows.
\end{proof}

The next lemma gives an estimate for the sum of values of a dilated floor function. We will use this estimate to prove the main Lemma \ref{lem:oAnn-oAnx-estimate} below.

%
%

\begin{lem}\label{lem:kernel-estimate}
For all real numbers $t$ and $u$ such that $1\le u\le t$, we have
$$
\sum_{j=1}^{\lfloor t\rfloor}\left\lfloor\frac{j}{u}\right\rfloor=t\left\lfloor\frac{t}{u}\right\rfloor-\frac{1}{2}u\left\lfloor\frac{t}{u}\right\rfloor^2-\frac{1}{2}u\left\lfloor\frac{t}{u}\right\rfloor+O\left(\frac{t}{u}\right).
$$
\end{lem}

\begin{proof} We write $\left\lfloor\frac{j}{u}\right\rfloor=\sum_{1\le k\le\frac{j}{u}}1$ and interchange the order of summation, obtaining
$$
\sum_{j=1}^{\lfloor t\rfloor}\left\lfloor\frac{j}{u}\right\rfloor
=\sum_{1\le j\le t}\bigg( \sum_{1\le k\le\frac{j}{u}}1\bigg) =\sum_{1\le k\le\frac{t}{u}}\bigg( \sum_{uk\le j\le t}1\bigg).
$$
The inner sum on the right counts the number of integers from $\lceil uk\rceil$ to $\lfloor t\rfloor$. Hence the above is
$$
\sum_{j=1}^{\lfloor t\rfloor}\left\lfloor\frac{j}{u}\right\rfloor=\sum_{1\le k\le\frac{t}{u}}\left(\lfloor t\rfloor-\lceil uk\rceil+1\right)=(\lfloor t\rfloor+1)\left\lfloor\frac{t}{u}\right\rfloor-\sum_{k=1}^{\left\lfloor\frac{t}{u}\right\rfloor}\lceil uk\rceil.
$$
By using the estimate $\lceil v\rceil=v+O(1)$, we obtain
\begin{align*}
\sum_{j=1}^{\lfloor t\rfloor}\left\lfloor\frac{j}{u}\right\rfloor&=(\lfloor t\rfloor+1)\left\lfloor\frac{t}{u}\right\rfloor-\sum_{k=1}^{\left\lfloor\frac{t}{u}\right\rfloor}uk+O\left(\frac{t}{u}\right)\\
&=t\left\lfloor\frac{t}{u}\right\rfloor-\frac{1}{2}u\left\lfloor\frac{t}{u}\right\rfloor^2-\frac{1}{2}u\left\lfloor\frac{t}{u}\right\rfloor+O\left(\frac{t}{u}\right)
\end{align*}
as desired.
\end{proof}

The following lemma gives an estimate for the complement sum $\oA(n,n)-\oA(n,x)$.

%
%

\begin{lem}\label{lem:oAnn-oAnx-estimate}
For all integers $n\ge2$ and real numbers $x$ such that $\sqrt{n-1}\le x\le n$, we have
\begin{equation}\label{eqn:oAnn-oAnx-asymp}
\oA(n,n)-\oA(n,x)=\int_x^n\left(\left\lfloor\frac{n}{u}\right\rfloor+\left\{\frac{n}{u}\right\}^2\right)u\log u\,du+O\left(\frac{n^2\log n}{x}\right).
\end{equation}
\end{lem}

\begin{proof}
We start from \eqref{eqn:oAnn-oAnx-formula} and apply Proposition \ref{prop:23n} to estimate each 
$\left(\oC(j,j)-\oC(j,x)\right)$:
\begin{align}\label{eqn:oAnn-oAnx-start2}
\oA(n,n)-\oA(n,x)&=\oB^c_{11}(n,x)-2\sum_{x\le j<n}\left(\oC(j,j)-\oC(j,x)\right)\nonumber\\
&=\oB^c_{11}(n,x)-2\sum_{x\le j<n}\int_x^j\left\lfloor\frac{j}{u}\right\rfloor\log u\,du+O\bigg(\frac{1}{x}\sum_{j=1}^nj\log j\bigg).
\end{align}
Now, we estimate each term on the right of \eqref{eqn:oAnn-oAnx-start}. 
To simplify the error term, by Lemma \ref{lem:Ln01-estimate}, we have
\begin{equation}\label{eqn:oAnn-oAnx-error}
\frac{1}{x}\sum_{j=1}^nj\log j=O\left(\frac{n^2\log n}{x}\right).
\end{equation}
The first term can be estimated by Lemma \ref{lem:52}:
\begin{align}\label{eqn:oAnn-oAnx-first}
\oB^c_{11}(n,x)&=\frac{1}{2}n^2(\log n)^2-\frac{1}{2}n^2(\log x)^2+O\left(\frac{n^2\log n}{x}\right)\nonumber\\
&=\int_x^n\frac{n^2\log u}{u}\,du++O\left(\frac{n^2\log n}{x}\right).
\end{align}
For the second term, we observe that $\left\lfloor\frac{j}{u}\right\rfloor=0$ for $0<j<u$. Hence
\begin{align}\label{eqn:oAnn-oAnx-second-1}
-2\sum_{x\le j<n}\int_x^j\left\lfloor\frac{j}{u}\right\rfloor\log u\,du&=-2\sum_{x\le j<n}\int_x^n\left\lfloor\frac{j}{u}\right\rfloor\log u\,du\nonumber\\
&=-2\int_x^n\bigg(\sum_{x\le j<n}\left\lfloor\frac{j}{u}\right\rfloor\bigg)\log u\,du.
\end{align}
The inner sum on the right of \eqref{eqn:oAnn-oAnx-second-1} can be estimated using Lemma \ref{lem:kernel-estimate}. If  $1\le j<x$ and $u \ge x$, then $0<j<u$, and then $\left\lfloor\frac{j}{u}\right\rfloor=0$. Hence $\sum_{1\le j<x}\left\lfloor\frac{j}{u}\right\rfloor=0$ and
\begin{align*}
\sum_{x\le j<n}\left\lfloor\frac{j}{u}\right\rfloor&=\sum_{j=1}^n\left\lfloor\frac{j}{u}\right\rfloor-\left\lfloor\frac{n}{u}\right\rfloor\\
&=n\left\lfloor\frac{n}{u}\right\rfloor-\frac{1}{2}u\left\lfloor\frac{n}{u}\right\rfloor^2-\frac{1}{2}u\left\lfloor\frac{n}{u}\right\rfloor+O\left(\frac{n}{u}\right)\\
&=n\left\lfloor\frac{n}{u}\right\rfloor-\frac{1}{2}u\left\lfloor\frac{n}{u}\right\rfloor^2-\frac{1}{2}u\left\lfloor\frac{n}{u}\right\rfloor+O\left(\frac{n}{x}\right).
\end{align*}
On inserting this into \eqref{eqn:oAnn-oAnx-second-1}, we obtain
\begin{equation}\label{eqn:oAnn-oAnx-second-2}
-2\sum_{x\le j<n}\int_x^j\left\lfloor\frac{j}{u}\right\rfloor\log u\,du=\int_x^n\left(-2n\left\lfloor\frac{n}{u}\right\rfloor+u\left\lfloor\frac{n}{u}\right\rfloor^2+u\left\lfloor\frac{n}{u}\right\rfloor\right)\log u\,du+O\left(\frac{n^2\log n}{x}\right).
\end{equation}
On inserting \eqref{eqn:oAnn-oAnx-error}, \eqref{eqn:oAnn-oAnx-first}, and \eqref{eqn:oAnn-oAnx-second-2} into \eqref{eqn:oAnn-oAnx-start2}, we obtain
\begin{align*}
\oA(n,n)-\oA(n,x)&=\int_x^n\left(\frac{n^2}{u}-2n\left\lfloor\frac{n}{u}\right\rfloor+u\left\lfloor\frac{n}{u}\right\rfloor^2+u\left\lfloor\frac{n}{u}\right\rfloor\right)\log u\,du+O\left(\frac{n^2\log n}{x}\right)\\
&=\int_x^n\left(\left(\frac{n}{u}-\left\lfloor\frac{n}{u}\right\rfloor\right)^2+\left\lfloor\frac{n}{u}\right\rfloor\right)u\log u\,du+O\left(\frac{n^2\log n}{x}\right)\\
&=\int_x^n\left(\left\lfloor\frac{n}{u}\right\rfloor+\left\{\frac{n}{u}\right\}^2\right)u\log u\,du
+O\left(\frac{n^2\log n}{x}\right)
\end{align*}
as desired.
\end{proof}

The next lemma shows that the main term in \eqref{eqn:oAnn-oAnx-asymp} can be written in the form $f\left(\frac{x}{n}\right)n^2\log n+g\left(\frac{x}{n}\right)n^2$.

%
%

\begin{lem}\label{lem:sec6-integral-eval1}
For all real numbers $n$ and $x$ such that $0<x\le n$, we have
\begin{equation}\label{eqn:sec6-main-term-evaluation}
\int_x^n\left(\left\lfloor\frac{n}{u}\right\rfloor+\left\{\frac{n}{u}\right\}^2\right)u\log u\,du
=n^2(\log n)\int_1^\frac{n}{x}\frac{\lfloor v\rfloor+\{v\}^2}{v^3}\,dv-n^2\int_1^\frac{n}{x}\frac{\lfloor v\rfloor+\{v\}^2}{v^3}\log v\,dv.
\end{equation}
\end{lem}

\begin{proof}
By the substitution $v=\frac{n}{u}$, we see that
\begin{align*}
\int_x^n\left(\left\lfloor\frac{n}{u}\right\rfloor+\left\{\frac{n}{u}\right\}^2\right)u\log u\,du
&=\int_\frac{n}{x}^1\left(\lfloor v\rfloor+\{v\}^2\right)\frac{n}{v}\left(\log\frac{n}{v}\right)\left(-\frac{n}{v^2}\right)\,dv\\
&=n^2\int_1^\frac{n}{x}\frac{\lfloor v\rfloor+\{v\}^2}{v^3}\log\frac{n}{v}\,dv.
\end{align*}
Since $\log\frac{n}{v}=\log n-\log v$, we obtain
$$
\int_x^n\left(\left\lfloor\frac{n}{u}\right\rfloor+\left\{\frac{n}{u}\right\}^2\right)u\log u\,du=n^2(\log n)\int_1^\frac{n}{x}\frac{\lfloor v\rfloor+\{v\}^2}{v^3}\,dv-n^2\int_1^\frac{n}{x}\frac{\lfloor v\rfloor+\{v\}^2}{v^3}\log v\,dv
$$
as desired.
\end{proof}

To evaluate the integrals on the right of \eqref{eqn:sec6-main-term-evaluation}, we use the following lemma.

%
%

\begin{lem}\label{lem:integral-calculation-EM}
Suppose that $f$ is a twice differentiable function with continuous second derivative on the interval $[1,\infty)$. Then for all real numbers $\beta\ge1$,
\begin{equation}\label{eqn:integral-calculation-EM}
\frac{1}{2}\int_1^\beta\left(\lfloor v\rfloor+\{v\}^2\right)f''(v)\,dv=\int_1^\beta f(v)\,dv-\sum_{b=2}^{\lfloor\beta\rfloor}f(b)-\{\beta\}f(\beta)+\frac{1}{2}\left(\lfloor\beta\rfloor+\{\beta\}^2\right)f'(\beta)-\frac{1}{2}f'(1).
\end{equation}
\end{lem}

\begin{proof}
By the Euler--Maclaurin summation formula (cf. \cite[Theorem~B.5]{MV07}),
\begin{align*}
\sum_{b=2}^{\lfloor\beta\rfloor}f(b)=&\int_1^\beta f(v)\,dv-\left(\{\beta\}-\frac{1}{2}\right)f(\beta)-\frac{1}{2}f(1)+\frac{1}{2}\left(\{\beta\}^2-\{\beta\}+\frac{1}{6}\right)f'(\beta)\\
&-\frac{1}{12}f'(1)-\frac{1}{2}\int_1^\beta\left(\{v\}^2-\{v\}+\frac{1}{6}\right)f''(v)\,dv.
\end{align*}
Rearranging the terms, we obtain
\begin{align}
\frac{1}{2}\int_1^\beta\left(\{v\}^2-\{v\}+\frac{1}{6}\right)f''(v)\,dv=&\int_1^\beta f(v)\,dv-\sum_{b=2}^{\lfloor\beta\rfloor}f(b)-\left(\{\beta\}-\frac{1}{2}\right)f(\beta)\nonumber\\
&+\frac{1}{2}\left(\{\beta\}^2-\{\beta\}+\frac{1}{6}\right)f'(\beta)-\frac{1}{2}f(1)-\frac{1}{12}f'(1).\label{eqn:Euler-Maclaurin-2}
\end{align}
On the other hand, we use integration by parts to see that
\begin{align}
\frac{1}{2}\int_1^\beta\left(v-\frac{1}{6}\right)f''(v)\,dv&=\frac{1}{2}\left(v-\frac{1}{6}\right)f'(v)\bigg|_{v=1}^\beta-\frac{1}{2}\int_1^\beta f'(v)\,dv\nonumber\\
&=-\frac{1}{2}f(\beta)+\frac{1}{2}\left(\beta-\frac{1}{6}\right)f'(\beta)+\frac{1}{2}f(1)-\frac{5}{12}f'(1).\label{eqn:by-parts}
\end{align}
Adding \eqref{eqn:Euler-Maclaurin-2} and \eqref{eqn:by-parts}, we obtain \eqref{eqn:integral-calculation-EM}.
\end{proof}

We apply Lemma \ref{lem:integral-calculation-EM} to evaluate the integrals on the right of \eqref{eqn:sec6-main-term-evaluation}.

%
%

\begin{lem}\label{lem:sec6-integral-eval2}
For all real numbers $\alpha$ such that $0<\alpha\le1$, we have
\begin{eqnarray}\label{eqn:oAnn-oAnx-integral-f-new}
\int_1^\frac{1}{\alpha}\frac{\lfloor v\rfloor+\{v\}^2}{v^3}\,dv&=&\frac{3}{2}-\left(H_{\lfloor\frac{1}{\alpha}\rfloor}-\log\frac{1}{\alpha}\right)-\alpha\left(\left\{\frac{1}{\alpha}\right\}+\frac{1}{2}\right)\nonumber\\
&&-\alpha^2\left(\frac{1}{2}\left\{\frac{1}{\alpha}\right\}^2-\frac{1}{2}\left\{\frac{1}{\alpha}\right\}\right),
\end{eqnarray}
\begin{eqnarray}\label{eqn:oAnn-oAnx-integral-g-new}
\int_1^\frac{1}{\alpha}\frac{\lfloor v\rfloor+\{v\}^2}{v^3}\log v\,dv&=&\frac{7}{4}-\frac{3}{2}\left(H_{\lfloor\frac{1}{\alpha}\rfloor}-\log\frac{1}{\alpha}\right)-\left(J\left(\frac{1}{\alpha}\right)-\frac{1}{2}\left(\log\frac{1}{\alpha}\right)^2\right)\nonumber\\
&&-\left(\alpha\log\frac{1}{\alpha}\right)\left(\left\{\frac{1}{\alpha}\right\}+\frac{1}{2}\right)-\alpha\left(\frac{3}{2}\left\{\frac{1}{\alpha}\right\}+\frac{1}{4}\right)\nonumber\\
&&-\left(\alpha^2\log\frac{1}{\alpha}\right)\left(\frac{1}{2}\left\{\frac{1}{\alpha}\right\}^2-\frac{1}{2}\left\{\frac{1}{\alpha}\right\}\right)\nonumber\\
&&-\alpha^2\left(\frac{1}{4}\left\{\frac{1}{\alpha}\right\}^2-\frac{1}{4}\left\{\frac{1}{\alpha}\right\}\right).
\end{eqnarray}
\end{lem}

\begin{proof}
For \eqref{eqn:oAnn-oAnx-integral-f-new}, apply Lemma \ref{lem:integral-calculation-EM} with $f(v)=\frac{1}{v}$ and $\beta=\frac{1}{\alpha}$:
$$
\int_1^\frac{1}{\alpha}\frac{\lfloor v\rfloor+\{v\}^2}{v^3}\,dv=\log\frac{1}{\alpha}-\left(H_{\lfloor\frac{1}{\alpha}\rfloor}-1\right)-\alpha\left\{\frac{1}{\alpha}\right\}-\frac{1}{2}\alpha^2\left(\left\lfloor\frac{1}{\alpha}\right\rfloor+\left\{\frac{1}{\alpha}\right\}^2\right)+\frac{1}{2}.
$$
Replacing $\left\lfloor\frac{1}{\alpha}\right\rfloor+\left\{\frac{1}{\alpha}\right\}^2$ by $\frac{1}{\alpha}-\left\{\frac{1}{\alpha}\right\}+\left\{\frac{1}{\alpha}\right\}^2$ and rearranging the terms, we obtain \eqref{eqn:oAnn-oAnx-integral-f-new}.\\

For \eqref{eqn:oAnn-oAnx-integral-g-new}, apply Lemma \ref{lem:integral-calculation-EM} with $f(v)=\frac{3}{2v}+\frac{\log v}{v}$ and $\beta=\frac{1}{\alpha}$:
\begin{align*}
\int_1^\frac{1}{\alpha}\frac{\lfloor v\rfloor+\{v\}^2}{v^3}\log v\,dv=&\left(\frac{3}{2}\log\frac{1}{\alpha}+\frac{1}{2}\left(\log\frac{1}{\alpha}\right)^2\right)-\left(\frac{3}{2}H_{\lfloor\frac{1}{\alpha}\rfloor}+J\left(\frac{1}{\alpha}\right)-\frac{3}{2}\right)\\
&-\alpha\left\{\frac{1}{\alpha}\right\}\left(\frac{3}{2}+\log\frac{1}{\alpha}\right)-\alpha^2\left(\left\lfloor\frac{1}{\alpha}\right\rfloor+\left\{\frac{1}{\alpha}\right\}^2\right)\left(\frac{1}{4}+\frac{1}{2}\log\frac{1}{\alpha}\right)+\frac{1}{4}.
\end{align*}
Replacing $\left\lfloor\frac{1}{\alpha}\right\rfloor+\left\{\frac{1}{\alpha}\right\}^2$ by $\frac{1}{\alpha}-\left\{\frac{1}{\alpha}\right\}+\left\{\frac{1}{\alpha}\right\}^2$ and rearranging the terms, we obtain \eqref{eqn:oAnn-oAnx-integral-g-new}.
\end{proof}

%
%

\subsection{Proof of Theorem \ref{thm:oAnx}}
\label{subsec:63n}

We combine results in the previous subsection to obtain an estimate for $\oA(n,x)$ as stated in Theorem \ref{thm:oAnx}.

\begin{proof}[Proof of Theorem \ref{thm:oAnx}]
Combining Theorem \ref{thm:oAn} and Lemma \ref{lem:oAnn-oAnx-estimate}, which estimate $\oA(n,n)$ and $\oA(n,n)-\oA(n,x)$ respectively, we obtain an estimate for $\oA(n,x)$:
\begin{eqnarray*}
\oA(n,x)&=&\left(\frac{3}{2}-\gamma\right)n^2\log n+\left(\frac{3}{2}\gamma+\gamma_1-\frac{7}{4}\right)n^2\\
&&-\int_x^n\left(\left\lfloor\frac{n}{u}\right\rfloor+\left\{\frac{n}{u}\right\}^2\right)u\log u\,du+O\left(n^{3/2}\log n\right).
\end{eqnarray*}
The integral on the right can be evaluated using Lemma \ref{lem:sec6-integral-eval1}:
\begin{align*}
\oA(n,x)=&\,\left(\frac{3}{2}-\gamma\right)n^2\log n+\left(\frac{3}{2}\gamma+\gamma_1-\frac{7}{4}\right)n^2\\
&-n^2(\log n)\int_1^\frac{n}{x}\frac{\lfloor v\rfloor+\{v\}^2}{v^3}\,dv+n^2\int_1^\frac{n}{x}\frac{\lfloor v\rfloor+\{v\}^2}{v^3}\log v\,dv+O\left(n^{3/2}\log n\right)\\
=&\,\oA_0(n,x)n^2\log n+\oA_1(n,x)n^2+O\left(n^{3/2}\log n\right),
\end{align*}
where the functions $\oA_0(n,x)$ and $\oA_1(n,x)$ are given by
\begin{equation}\label{eqn:oA0-integral}
\oA_0(n,x):=\left(\frac{3}{2}-\gamma\right)-\int_1^\frac{n}{x}\frac{\lfloor v\rfloor+\{v\}^2}{v^3}\,dv,
\end{equation}
\begin{equation}\label{eqn:oA1-integral}
\oA_1(n,x):=\left(\frac{3}{2}\gamma+\gamma_1-\frac{7}{4}\right)+\int_1^\frac{n}{x}\frac{\lfloor v\rfloor+\{v\}^2}{v^3}\log v\,dv.
\end{equation}
It remains to show that \eqref{eqn:oA0-integral} is equivalent to \eqref{eqn:oA0-thm} and that \eqref{eqn:oA1-integral} is equivalent to \eqref{eqn:oA1-thm}. To that end, we apply Lemma \ref{lem:sec6-integral-eval2} with $\alpha=\frac{x}{n}$ to evaluate the integrals in \eqref{eqn:oA0-integral} and \eqref{eqn:oA1-integral}. 
We obtain 
$$
\oA_0(n,x) =\left( H_{\lfloor \frac{n}{x} \rfloor} - \log \frac{n}{x} -\gamma \right) 
+ \frac{x}{n} \left( \left\{ \frac{n}{x} \right\} + \frac{1}{2} \right) -\left(\frac{x}{n}\right)^2
\left( \frac{1}{2}\left\{\frac{n}{x}\right\}^2-\frac{1}{2}\left\{\frac{n}{x}\right\}\right),
$$
\begin{eqnarray*} 
\oA_1(n,x) &= &
-\frac{3}{2}\left(H_{\lfloor\frac{n}{x}\rfloor}-\log\frac{n}{x}-\gamma\right)
-\left(J\left(\frac{n}{x}\right)-\frac{1}{2}\left(\log\frac{n}{x}\right)^2 -\gamma_1\right)\nonumber\\
&&-\left(\frac{x}{n}\log\frac{n}{x}\right)\left(\left\{\frac{n}{x}\right\}+\frac{1}{2}\right)
-\frac{x}{n} \left(\frac{3}{2}\left\{\frac{n}{x}\right\}+\frac{1}{4}\right)\nonumber\\
&&-\left(\frac{x}{n}\right)^2\left(\log\frac{n}{x}\right)\left(\frac{1}{2}\left\{\frac{n}{x}\right\}^2
-\frac{1}{2}\left\{\frac{n}{x}\right\}\right)
-\left(\frac{x}{n}\right)^2\left(\frac{1}{4}\left\{\frac{n}{x}\right\}^2-\frac{1}{4}\left\{\frac{n}{x}\right\}\right).
\end{eqnarray*}
Replacing $\left\{\frac{n}{x}\right\}$ by $\frac{n}{x}-\left\lfloor\frac{n}{x}\right\rfloor$ and rearranging the terms, we obtain the formulas \eqref{eqn:oA0-thm} and \eqref{eqn:oA1-thm}.\\

Finally, for integers $n\ge2$ and real $x\in\left[1,\sqrt{n}\right]$, we have
\begin{align*}
\oA(n,x)&=\sum_{2\le b\le x} \frac{2}{b-1}S_b(n)\log b\\
&\le\sum_{2\le b\le x}n\log n <xn\log n\\
 &\le n^{3/2}\log n,
\end{align*}
where the bound of Lemma \ref{lem:dbn-Sbn-bound} for $S_b(n)$ was used in the first inequality.
We have obtained \eqref{eqn:oAnx-bound2}.
\end{proof}

%
%

\subsection{Proof of Theorem \ref{thm:oAnx-cor}}
\label{subsec:64n}

\begin{proof}[Proof of Theorem \ref{thm:oAnx-cor}] 
The result for the range $x\in [\sqrt{n}, n]$ 
follows from Theorem \ref{thm:oAnx} on choosing  $x= \alpha n$ and simplifying.
For the range $x\in \left[1, \sqrt{n}\right]$ we use the
final estimate \eqref{eqn:oAnx-bound2}.
\end{proof} 


\begin{rem}\label{rem:38}
The function $f_\oA(\alpha)$ has  $f_\oA(1)=\frac{3}{2} -\gamma$, and has  $\lim_{\alpha \to 0^+} f_\oA(\alpha)=0$   since $ H_{\lfloor \frac{1}{\alpha}\rfloor}- \log \frac{1}{\alpha}  \to \gamma$
as $\alpha \to 0^+$.
\end{rem}

%
%

\section{Estimates for partial factorizations $\oH(n,x)$}\label{sec:oHnx}

We deduce asymptotics of $\oH(n,x)$.

%
%

\begin{thm}\label{thm:oHnx}
Let $\oH(n, x) = \prod_{b=2}^{\lfloor x\rfloor} b^{\Gnu(n,b)}$. Then for all integers $n\ge2$ and real $x\in\left[\sqrt{n},n\right]$,
\begin{equation}\label{eqn:oHnx-main-two-var1} 
\log \oH(n, x ) = \oH_0(n,x)\,n^2\log n + \oH_1(n,x) n^2
+  O \left(n^{3/2}  \log n \right),
\end{equation}
where the functions $\oH_0(n,x)$ and $\oH_1(n,x)$ only depend on $\frac{x}{n}$ and are given by
\begin{equation}\label{eqn:oH0nx}
\oH_0(n,x) := \frac{1}{2}  +\frac{1}{2} \left(\frac{x}{n}\right)^2 \left\lfloor \frac{n}{x}\right\rfloor \left\lfloor \frac{n}{x}+1\right\rfloor
- \frac{x}{n} \left\lfloor \frac{n}{x} \right\rfloor
\end{equation}
and 
\begin{align}
\oH_1(n,x) := \left( \frac{1}{2} \gamma - \frac{3}{4} \right) - \frac{1}{2} \left( H_{\lfloor \frac{n}{x}\rfloor} - \log \frac{n}{x} \right)-\frac{1}{2}\log\frac{n}{x} -\frac{1}{2}\left(\log \frac{n}{x} \right) \left(\frac{x}{n}\right)^2 \left\lfloor \frac{n}{x} \right\rfloor \left\lfloor \frac{n}{x}+1 \right\rfloor  \nonumber\\
+\left(\log \frac{n}{x} \right)\frac{x}{n} \left\lfloor \frac{n}{x} \right\rfloor 
- \frac{1}{4} \left(\frac{x}{n}\right)^2 \left\lfloor \frac{n}{x} \right\rfloor \left\lfloor \frac{n}{x} +1\right\rfloor   + \frac{x}{n} \left\lfloor \frac{n}{x} \right\rfloor. \label{eqn:oH1nx}
\end{align}
Moreover, for all integers $n\ge2$ and real $x\in\left[1, \sqrt{n}\right]$,
\begin{equation} \label{eqn:oHnx-main-two-var2} 
\log \oH(n, x ) = O \left( n^{3/2}  \log n\right).
\end{equation}
\end{thm} 

\begin{proof}
Recall from \eqref{eqn:oHABx} the identity 
$$
\log \oH(n, x) = \oA(n, x) - \oB(n,x).
$$
The result  \eqref{eqn:oHnx-main-two-var1}
follows for the range $x \in [\sqrt{n}, n]$  by inserting  the formulas \eqref{eqn:oAnx-bound} in Theorem \ref{thm:oAnx}   and \eqref{eqn:oBnx-bound} in Theorem \ref{thm:oBnx}. 
The formula \eqref{eqn:oHnx-main-two-var2} in the range $x \in \left[1, \sqrt{n}\right]$ follows from the
corresponding range bounds in Theorems \ref{thm:oAnx} and \ref{thm:oBnx}.  
\end{proof}

%
%

\subsection{Proof of Theorem \ref{thm:oHnx-main}} \label{subsec:81}

\begin{proof}[Proof of Theorem \ref{thm:oHnx-main}]
The theorem follows from Theorem \ref{thm:oHnx}  on choosing $x=\alpha n$ and simplifying.
The $O$-constant  in the remainder term is absolute for the range $\frac{1}{\sqrt{n}}  \le \alpha \le 1.$ 
Here $\frac{n}{x}= \frac{1}{\alpha}$. 
\end{proof}

%
%

\subsection{Properties of the limit scaling function $g_{\oH}(\alpha)$} \label{subsec:82}

%
%

\begin{lem}\label{lem:93}
Let $g_{\oH} (\alpha)$  be the  function 
 defined for all $0 < \alpha \le1$ by
\begin{eqnarray*}
g_{\oH}(\alpha) &=&  \bigg( \frac{1}{2} \gamma - \frac{3}{4} \bigg) -\frac{1}{2}\bigg( H_{\lfloor\frac{1}{\alpha}\rfloor}- \log \frac{1}{\alpha} \bigg) 
 + \bigg(\log \frac{1}{\alpha} \bigg)\bigg( -\frac{1}{2} - \frac{1}{2} \alpha^2 \bigg\lfloor \frac{1}{\alpha} \bigg\rfloor \bigg\lfloor \frac{1}{\alpha} +1\bigg\rfloor + \alpha \bigg\lfloor \frac{1}{\alpha} \bigg\rfloor \bigg)  
 \nonumber \\
&&- \frac{1}{4} \alpha^2 \bigg\lfloor \frac{1}{\alpha} \bigg\rfloor \bigg\lfloor \frac{1}{\alpha} +1\bigg\rfloor + \alpha \bigg\lfloor \frac{1}{\alpha} \bigg\rfloor.\label{eqn:oHnx-parametrized-3}
\end{eqnarray*}
It has the following properties:
\begin{enumerate}
\item[\emph{(1)}]
$g_{\oH}(\alpha)$ continuously extends to the domain $[0,1]$, taking the value 
$g_{\oH}(0)=0$.
\item[\emph{(2)}]
Let $\tilde{g}(\alpha)$ be defined for all $0 < \alpha \le 1$ by
\begin{equation}\label{eqn:g-tilde-def}
\tilde{g}(\alpha) = - \frac{1}{2} \psi\left(\frac{1}{\alpha}+1\right) + \frac{1-\alpha}{2}\log \frac{1}{\alpha} - \frac{\alpha}{4},
\end{equation}
where $\psi(x) = \frac{\Gamma'}{\Gamma}(x)$ is the digamma function.  It is real-analytic on $(0,1]$ and
continuously extends to $[0,1]$, taking the value $\tilde{g}(0)=0$.
It  interpolates $g_{\oH}(\alpha)$ at reciprocal integer values; i.e., $\tilde{g}\left(\frac{1}{j}\right)=g_{\oH}\left(\frac{1}{j}\right)$ for all integers $j \ge 1$. 
\item[\emph{(3)}]
The function $\tilde{g}(\alpha)$ is strictly convex on $0 \le \alpha \le 1.$
The function $g_{\oH}(\alpha)$ is strictly concave on each interval $I_j=\left[\frac{1}{j+1}, \frac{1}{j}\right]$
for each integer $j \ge 2$.
\item[\emph{(4)}]
We have
\begin{equation}\label{eqn:ineq-tilde-g}
\tilde{g}(\alpha) \le g_{\oH}(\alpha) \quad \mbox{for all} \quad  0 \le \alpha \le 1.
\end{equation}

\end{enumerate}
\end{lem}

\begin{proof}
(1) Let $j\ge1$ be an integer, and suppose that $\frac{1}{j+1}<\alpha\le\frac{1}{j}$. Then $\left\lfloor\frac{1}{\alpha}\right\rfloor=j$, and
\begin{eqnarray}
g_{\oH}(\alpha) &=& \bigg(\frac{1}{2}\gamma-\frac{3}{4}\bigg)-\frac{1}{2}\bigg(H_j-\log\frac{1}{\alpha}\bigg)+\bigg(\log\frac{1}{\alpha}\bigg)\bigg(-\frac{1}{2}-\frac{1}{2}\alpha^2j(j+1)+\alpha j\bigg)\nonumber\\
&&-\frac{1}{4}\alpha^2j(j+1)+\alpha j.\label{eqn:oHnx-parametrized-3-j}
\end{eqnarray}
We first show the continuity of $g_{\oH}(\alpha)$ at transition points $\alpha=\frac{1}{2},\frac{1}{3},\frac{1}{4},\dots$. By \eqref{eqn:oHnx-parametrized-3-j}, it follows that
\begin{equation}\label{eqn:g_oH(1/j)}
\lim_{\varepsilon\rightarrow0^+}g_{\oH}\left(\frac{1}{j}-\varepsilon\right)=g_{\oH}\left(\frac{1}{j}\right)=-\frac{1}{2}\left(-\gamma+H_j\right)+\bigg(\frac{1}{2}-\frac{1}{2j}\bigg)\log j-\frac{1}{4j}
\end{equation}
and that
$$
\lim_{\varepsilon\rightarrow0^+}g_{\oH}\left(\frac{1}{j+1}+\varepsilon\right)=-\frac{1}{2}\left(-\gamma+H_j\right)+\frac{j}{2(j+1)}\log(j+1)-\frac{3}{4(j+1)}.
$$
Replacing $j$ by $j-1$ and using the identity $H_{j-1}=H_j-\frac{1}{j}$, we obtain
$$
\lim_{\varepsilon\rightarrow0^+}g_{\oH}\left(\frac{1}{j}+\varepsilon\right)=-\frac{1}{2}\left(-\gamma+H_j\right)+\bigg(\frac{1}{2}-\frac{1}{2j}\bigg)\log j-\frac{1}{4j}=g_{\oH}\left(\frac{1}{j}\right).
$$
Thus $g_{\oH}(\alpha)$ is continuous at $\alpha=\frac{1}{j}$ for all integers $j\ge2$.

By \eqref{eqn:oHnx-parametrized-3-j}, we see that $g_{\oH}(\alpha)$ is continuous on $\left(\frac{1}{j+1},\frac{1}{j}\right]$ for each integer $j\ge1$. So it remains to show that $\lim_{\alpha\rightarrow0^+}g_{\oH}(\alpha)=0$. Using the identity $\left\lfloor\frac{1}{\alpha}\right\rfloor=\frac{1}{\alpha}-\left\{\frac{1}{\alpha}\right\}$, we may rewrite $g_{\oH}(\alpha)$ as
\begin{eqnarray}
g_{\oH}(\alpha)&=&-\frac{1}{2}\bigg(H_{\lfloor\frac{1}{\alpha}\rfloor}-\log\frac{1}{\alpha}-\gamma\bigg)+\bigg(-\frac{1}{2}+\frac{1}{2}\left\{\frac{1}{\alpha}\right\}\alpha-\frac{1}{2}\left\{\frac{1}{\alpha}\right\}^2\alpha\bigg)\alpha\log\frac{1}{\alpha}\nonumber\\
&&-\frac{1}{4}\alpha-\frac{1}{2}\left\{\frac{1}{\alpha}\right\}\alpha+\frac{1}{4}\left\{\frac{1}{\alpha}\right\}\alpha^2-\frac{1}{4}\left\{\frac{1}{\alpha}\right\}^2\alpha^2,\label{eqn:oHnx-parametrized-3-fractional}
\end{eqnarray}
It follows that that $\lim_{\alpha\rightarrow0^+}g_{\oH}(\alpha)=0$, noting that
$H_{\lfloor\frac{1}{\alpha}\rfloor}-\log\frac{1}{\alpha}-\gamma \to 0$ as $\alpha \to 0^{+}$ by Lemma \ref{lem:21n}.

(2) The function $\tilde{g}(\alpha)$ is real-analytic on the interval $(0,1]$. 
Let $j\ge1$ be an integer. We first show that $\tilde{g}(\alpha)$ interpolates $g_{\oH}(\alpha)$ at $\alpha=\frac{1}{j}$. It is known that
\begin{equation}
\psi(j+1)=-\gamma+H_j,
\end{equation}
where $H_j$ is the $j$-th harmonic number; see, e.g., \cite[(5.4.14)]{NIST:DLMF} and \cite[(6.3.2)]{AS:64}. Substituting $\alpha=\frac{1}{j}$ in \eqref{eqn:g-tilde-def} and using \eqref{eqn:g_oH(1/j)}, we obtain
$$
\tilde{g}\left(\frac{1}{j}\right)=-\frac{1}{2}\left(-\gamma+H_j\right)+\bigg(\frac{1}{2}-\frac{1}{2j}\bigg)\log j-\frac{1}{4j}=g_{\oH}\left(\frac{1}{j}\right).
$$

Now, we show that $\lim_{\alpha\rightarrow0^+}\tilde{g}(\alpha)=0$. It is known that
$$
\psi(z+1)=\psi(z)+\frac{1}{z}=\log z+O\left(\frac{1}{z}\right)\quad\text{as}\quad z\rightarrow+\infty;
$$
see, e.g., \cite[(5.5.2)]{NIST:DLMF} and \cite[(5.11.2) \& \S5.11(ii)]{NIST:DLMF}. It follows from \eqref{eqn:g-tilde-def} that
$$
\tilde{g}(\alpha)=-\frac{1}{2}\alpha\log\frac{1}{\alpha}+O(\alpha)\quad\text{as}\quad\alpha\rightarrow0^+.
$$
So $\tilde{g}(\alpha)$ tends to $0$ as $\alpha\rightarrow0^+$.

(3) We first prove the convexity of $\tilde{g}(\alpha)$. Consider $0<\alpha<1$. From \eqref{eqn:g-tilde-def}, it follows that,
with $'$ denoting $\frac{d}{d\alpha}$ on the left, 
$$
\tilde{g}''(\alpha)=-\frac{1}{\alpha^3}\psi'\left(\frac{1}{\alpha}+1\right)-\frac{1}{2\alpha^4}\psi''\left(\frac{1}{\alpha}+1\right)+\frac{1}{2\alpha^2}+\frac{1}{2\alpha},
$$
but by convention $\psi'(z) := \frac{d}{dz} \psi(z)$ and $\psi''(z) := \frac{d^2}{dz^2} \psi(z)$
on the right side. 
The polygamma functions $\psi'$ and $\psi''$ can be expressed in partial fractions as
$$
\psi'\left(z\right)=\sum_{k=0}^\infty\frac{1}{\left(z+k\right)^2}\quad\text{and}\quad\psi''\left(z\right)=
-2\sum_{k=0}^\infty\frac{1}{\left(z+k\right)^3},
$$
with $'$ denoting $\frac{d}{dz}$; see, e.g., 
\cite[(6.4.10)]{AS:64}. So, with $'$ denoting $\frac{d}{d\alpha}$, taking $z= \frac{1}{\alpha}+1$, we obtain
$$
\tilde{g}''(\alpha)=-\frac{1}{\alpha^3}\sum_{k=1}^\infty\frac{k}{\left(\frac{1}{\alpha}+k\right)^3}+\frac{1}{2\alpha^2}+\frac{1}{2\alpha}.
$$
The following inequality is proved in \cite[Proposition 4.18]{BE:24}:
\begin{equation}\label{eqn:BE-ineq}
\sum_{k=1}^\infty\frac{k}{(a+k)^3}<\frac{1}{2a}\quad\text{for all}\quad a>0.
\end{equation}
Applying \eqref{eqn:BE-ineq} with $a=\frac{1}{\alpha}$ yields
$$
\tilde{g}''(\alpha)>\frac{1}{2\alpha}>0.
$$
So $\tilde{g}(\alpha)$ is convex on $[0,1]$.

Now, we prove that $g_{\oH}(\alpha)$ is concave on $I_j=\left[\frac{1}{j+1},\frac{1}{j}\right]$ for each integer $j\ge2$. Fix an integer $j\ge2$, and consider $\frac{1}{j+1}<\alpha<\frac{1}{j}$. Then \eqref{eqn:oHnx-parametrized-3-j} holds. Differentiating twice yields
$$
g_{\oH}''(\alpha)=j\left(j+1-\frac{1}{\alpha}-(j+1)\log\frac{1}{\alpha}\right),
$$
from which we observe that $g_{\oH}''(\alpha)$ is increasing. Since $j\ge2$, it follows that
$$
g_{\oH}''(\alpha)\le\lim_{\varepsilon\rightarrow0^+}g_{\oH}''\left(\frac{1}{j}-\varepsilon\right)=j(1-(j+1)\log j)\le j(1-3\log2)<0.
$$
So $g_{\oH}(\alpha)$ is concave on $I_j$ for each $j\ge2$. One may check $g_{\oH}(\alpha)$ is not                                                                                                                                                                                                                                                                                                                                                     concave on $I_1$, it has an inflection point.

(4) The inequality \eqref{eqn:ineq-tilde-g} on intervals $I_j$ for all $j\ge2$ follows from item (3). The functions $g_{\oH} (\alpha)$ and $\tilde{g}(\alpha)$ agree at the endpoints of the interval $I_j$ by  item (2).
The convexity of   $\tilde{g}(\alpha)$ on $I_j$ 
says its graph lies on or below the line segment connecting the endpoints $\alpha= \frac{1}{j+1}$ and $\alpha= \frac{1}{j}$ in its graph, 
while the concavity of $g_{\oH}(\alpha)$
puts it on or above this line segment.

The inequality \eqref{eqn:ineq-tilde-g} holds at $\alpha=0$, since $\tilde{g}(0)=0=g_{\oH}(0)$. It also holds at $\alpha=1$, by item (2). So it suffices to prove the inequality assuming that $\frac{1}{2}<\alpha<1$. It follows from \eqref{eqn:oHnx-parametrized-3-j} that
$$
g_{\oH}(\alpha)=-\frac{1}{2}\alpha^2+\alpha+\frac{1}{2}\gamma-\frac{5}{4}+\left(-\alpha^2+\alpha\right)\log\frac{1}{\alpha}.
$$
Subtracting by \eqref{eqn:g-tilde-def} yields
\begin{equation}\label{eqn:diff-g(alpha)}
g_{\oH}(\alpha)-\tilde{g}(\alpha)=\frac{1}{2}\psi\left(\frac{1}{\alpha}+1\right)-\frac{1}{2}\alpha^2+\frac{5}{4}\alpha+\frac{1}{2}\gamma-\frac{5}{4}+\left(-\alpha^2+\frac{3}{2}\alpha-\frac{1}{2}\right)\log\frac{1}{\alpha}.
\end{equation}
In \cite[Theorem 7]{Alzer:97}, the following inequality is given:
\begin{equation}\label{eqn:Alzer-ineq}
\psi(a+1)-\psi(a+s)>\frac{1-s}{a+s}\quad\text{for all}\quad a>0\quad\text{and}\quad0<s<1. 
\end{equation}
Applying \eqref{eqn:Alzer-ineq} with $a=\frac{1}{\alpha}$ and $s=2-\frac{1}{\alpha}$ and using $\psi(2)=-\gamma+1$, we obtain
$$
\psi\left(\frac{1}{\alpha}+1\right)>-\gamma+\frac{1}{2}+\frac{1}{2\alpha}.
$$
Putting this inequality in \eqref{eqn:diff-g(alpha)} yields
\begin{align*}
g_{\oH}(\alpha)-\tilde{g}(\alpha)&>-\frac{1}{2}\alpha^2+\frac{5}{4}\alpha-1+\frac{1}{4\alpha}+\left(-\alpha^2+\frac{3}{2}\alpha-\frac{1}{2}\right)\log\frac{1}{\alpha}\\
&=(1-\alpha)\left(\alpha-\frac{1}{2}\right)\left(\frac{1}{2}-\frac{1}{2\alpha}+\log\frac{1}{\alpha}\right)\\
&>0,
\end{align*}
where the last inequality is true because $\log(1+t)>\frac{1}{2}t$ for all $0<t<1$, which is \eqref{eqn:ineq-tilde-g}.
\end{proof}

%
%

\section{Concluding remarks}\label{sec:8}

\begin{enumerate}
\item[(1)]
 One may view  the general definition of generalized binomial products \eqref{eqn:oHn0}  as a kind of  discrete (multiplicative)
 ``integration"  operation; after taking logarithms, we sum over  terms indexed by $b \ge2$. The 
 smoothing aspect of this  operation  is evident in the 
 remainder term estimate: there are 
  of unconditional estimates giving a power-savings remainder term; the Riemann hypothesis is
  not needed.  We do not know of an  analogous notion of  multiplicative ``differencing"  that 
  recovers the sequence $\G_n$ from the sequence $\oH_n$.   One may need to know many  quantities beyond the sequence $\oH_n$
  to obtain a differencing-type reconstruction of the sequence $\G_n$.
 \item[(2)]
It may be possible to  iterate the process of defining generalized binomial coefficients
by  first completely factorizing $\oH_n$ into its prime factorization:
\begin{equation} \label{eqn:oH-gen}
\oH_n  = \prod_p p^{\ord_p(\oH_n)}.
\end{equation}
The resulting quantities  $\ord_p(\oH_n)$ may have 
(more complicated)  interpretations in terms of radix expansions. 
If so,   we may hope to define by analogy another
 product that computes the analogous statistics for all bases $b$
on the right side of \eqref{eqn:oH-gen}.
\item[(3)]
(Limit scaling functions)
A large class of limit scaling functions may occur 
in problems of this sort,  generalizing the limit function $f_{G}(\alpha)$
in  \cite{DL:22}, which equals $f_{\oH}(\alpha)$. This paper 
exhibited a new  limit scaling functions $g_{\oH}(\alpha)$.
It may be of interest to determine  and characterize the class of of such  scaling functions 
obtainable by iterated integral  constructions of this kind.
\item[(4)] 
It is an interesting question to heuristically account for the appearance and form of the secondary
limit scaling function $g_{\oH}( \alpha)$. We do not treat this question here,
but think the form of this function encodes information on
the distribution on the spacing of gaps between primes. 
\item[(5)]
We consider the problem of recovering $G(n,x)$ from $\oH(n,x)$, allowing $x$ to vary. 
The values  $\oH(n, x)$ for $1 \le x \le n$ contain full information 
about the partial binomial products $G(n,x)$ (which form a subset of terms of the product). 
We have $\frac{ \oH(n, b)}{\oH(n, b-1)} = b^{\Gnu(n,b)}$, 
and since $G(n,x)= \prod_{p \le x} p^{\Gnu(n,p)}$,
we have
$$
G(n, x) = \prod_{p \le x} \frac{ \oH(n, p)}{\oH(n, p-1)}.
$$
We raise the question of finding more natural ``inversion formulas" connecting $G(n,x)$ and $\oH(n,x)$, 
 without introducing primes, perhaps using all  
$\oH(m, y)$ for $1 \le y \le m \le n$. (Note that $\oH(m, y) = \oH(m,m)$ whenever $y \ge m$.)
Perhaps  additional extended binomial partial products must be introduced, which 
consider other extended binomial coefficients (as in Section \ref{subsubsec:212})
 associated to restricted sets $\sT \subseteq \NN$,
such as $\sT$ in arithmetic progressions.
 \item[(6)]
In a related paper (\cite{LY:24c}) two of the authors study behavior of the extended factorials
$n!_{\ZZ, \NN}$.  There is a smooth asymptotic formula for $\log n!_{\ZZ, \NN}$,
analogous to the initial terms of Stirling's formula.  It  implies for fixed $\alpha$  the
growth of (scaled)  individual binomial coefficients 
$\log \binom{n}{\lfloor \alpha n\rfloor}_{\ZZ, \NN}$
 has asymptotic growth with a smooth main term, and with a remainder term  $O \left( n^{3/2}\right)$. 
\end{enumerate}

\medskip

\noindent{\bf Acknowledgements}\\
Theorems \ref{thm:oBn}  and \ref{thm:oAn} were first established in Chapter 3.6 
of the PhD thesis of the first author (\cite{Du:2020}),
using \eqref{eqn:oHn0} as the definition.  The first author  thanks  Trevor Wooley for helpful comments. 
 The authors thank D. Harry Richman for
helpful comments, and for earlier versions of plots of the scaling functions.

\medskip

\section*{Declarations}

\noindent{\bf Funding}\\
Work of the first author was  partially supported by NSF-grant DMS-1701577.
Work of the  second author was partly supported by NSF-grant  DMS-1701576 and
 a Simons  Fellowship in Mathematics in 2019. Work of the third author was
partially supported by  NSF-grants DMS-1701576 and DMS-1701577.\medskip

%
%

\end{document}